\documentclass[a4paper]{amsart}
\usepackage[utf8]{inputenc}
\usepackage[ngerman, english]{babel}
\usepackage{amsmath,amsthm,amssymb,amsfonts}
\usepackage{mathrsfs}
\usepackage{ifthen}
\usepackage{enumerate}
\usepackage{comment}
\usepackage[pdfauthor={Martin Halla},pdftitle={Electromagnetic Stekloff eigenvalues:
existence and behavior in the selfadjoint case},
unicode,colorlinks=true,linkcolor=black,citecolor=black,urlcolor=black,pagebackref=false,breaklinks]{hyperref}
\usepackage{hyphsubst}
\usepackage{xcolor}
\usepackage{graphicx}
\usepackage{psfrag}
\usepackage{pst-all}

\newtheorem{theorem}{Theorem}[section]
\newtheorem{lemma}[theorem]{Lemma}

\newtheorem{proposition}[theorem]{Proposition}

\newtheorem{assumption}[theorem]{Assumption}

\newcommand{\nv}{\nu}

\DeclareMathOperator{\I}{I}

\newcommand{\dd}{\mathrm{d}}
\DeclareMathOperator{\diveps}{div\epsilon}

\newcommand{\ol}[1]{\overline{#1}}

\newcommand{\spl}{\langle}
\newcommand{\spr}{\rangle}
\newcommand{\bpm}{\begin{pmatrix}}
\newcommand{\epm}{\end{pmatrix}}

\renewcommand{\div}{\operatorname{div}}
\DeclareMathOperator{\curl}{curl}

\DeclareMathOperator{\tr}{tr}

\DeclareMathOperator{\spn}{span}

\renewcommand{\dim}{\operatorname{dim}}
\DeclareMathOperator{\ran}{ran}

\newcommand{\setC}{\mathbb{C}}

\newcommand{\setM}{\mathbb{M}}
\newcommand{\setN}{\mathbb{N}}

\newcommand{\setR}{\mathbb{R}}

\newcommand{\boldH}{\mathbf{H}}

\newcommand{\boldL}{\mathbf{L}}

\newcommand{\boldT}{\mathbf{T}}

\newcommand{\boldX}{\mathbf{X}}

\title[Existence of electromagnetic Stekloff eigenvalues]{
Electromagnetic Stekloff eigenvalues: existence and behavior in the selfadjoint case}
\author{Martin Halla}
\email{martin.halla@protonmail.com}
\subjclass[2010]{{35J25, 35R30, 35P99.}}
\keywords{Stekloff eigenvalues, spectral analysis, nondestructive testing.}
\date{September 2nd, 2019.}

\begin{document}
\begin{abstract}
In [Camano, Lackner, Monk, SIAM J.\ Math.\ Anal., Vol.\ 49, No.\ 6, pp.\ 4376-4401 (2017)] it was suggested to use
Stekloff eigenvalues for Maxwell equations as target signature for nondestructive testing via inverse scattering.
The authors recognized that in general the eigenvalues due not correspond to the spectrum of a compact operator and
hence proposed a modified eigenvalue problem with the desired properties. The Fredholmness and the approximation of
both problems were analyzed in [Halla, arXiv:1909.00689 (2019)].

The present work considers the original eigenvalue problem in the selfadjoint case. We report that apart for a countable
set of particular frequencies, the spectrum consists of three disjoint parts: The essential spectrum consisting of the
point zero, an infinite sequence of positive eigenvalues which accumulate only at infinity and an infinite sequence of
negative eigenvalues which accumulate only at zero.

The analysis is based on a representation of the operator as block operator. For small/big enough eigenvalue parameter
the Schur-complements with respect to different components can be build. For each Schur-complement the existence of an
infinite sequence of eigenvalues is proved via a fixed point technique similar to [Cakoni, Haddar, Applicable Analysis,
88:4, 475-493 (2009)]. The modified eigenvalue problem considered in the above references arises as limit of one of the
Schur-complements.
\end{abstract}
\maketitle

\section{Introduction}\label{sec:introduction}
Novel nondestructive evaluation methods based on inverse scattering~\cite{CakoniColton:06} give rise to a multitude of
new eigenvalue problems. Among these are so-called transmission eigenvalue problems~\cite{CakoniColtonHaddar:16}
and Stekloff eigenvalue problems~\cite{CakoniColtonMengMonk:16}. Not all of these eigenvalue problems fall into classes
which are covered in classical literature. Among the important questions on these eigenvalue problems are
\begin{itemize}
 \item Fredholm properties (which imply the discreteness of the spectrum),
 \item the existence of eigenvalues,
 \item properties of the eigenvalues
 \item and reliable computational approximations.
\end{itemize}
The electromagnetic Stekloff eigenvalue problem to find $(\lambda,u)$ so that
\begin{align*}
\curl\curl u-\omega^2\epsilon u &=0 \quad\text{in }\Omega,\\
\nv\times\curl u +\lambda\, \nv\times u\times \nv&=0\quad\text{at }\partial\Omega.
\end{align*}
was considered in the recent publication~\cite{CamanoLacknerMonk:17}. Therein the authors of~\cite{CamanoLacknerMonk:17}
considered the case that $\Omega$ is a ball and the material parameter $\epsilon$ is constant. For this setting they
proved the existence of two infinite sequences of eigenvalues, one converging to zero and one converging to infinity.
Consequently the eigenvalue problem can't be transformed to an eigenvalue problem for a compact operator.
This observation led the authors of~\cite{CamanoLacknerMonk:17} to discard the original eigenvalue problem and
to modify instead the boundary condition to obtain a different eigenvalue problem. The approximation of both eigenvalue
problems is discussed in the companion article \cite{Halla:19StekloffAppr} by means of \cite{Halla:19Tcomp}.\\

In this article we consider the original electromagnetic Stekloff eigenvalue problem in the selfadjoint case.
We give a complete description of the spectrum (see Proposition~\ref{prop:mainprop}): The spectrum consists of three
disjoint parts: The essential spectrum consisting of the point zero, an infinite sequence of positive eigenvalues which
accumulate only at infinity and an infinite sequence of negative eigenvalues which accumulate only at zero.

As a side result, we also analyze the spectrum of the modified electromagnetic Stekloff eigenvalue problem, see
Section~\ref{sec:conclusion}. Our analysis reveals that the modified eigenvalue problem arises as asymptotic limit
of the original eigenvalue problem for large spectral parameter. Though, this doesn't yield any non-trivial asymptotic
statement on the eigenvalues.\\

The remainder of this article is organized as follows. In Section~\ref{sec:generalsetting} we set our notation and
formulate our assumptions on the domain and the material parameters. We also recall some classic regularity, embedding
and decomposition results which will be essential for our analysis. In most cases the respective references don't apply
directly to our setting and hence we formulate adapted variants.
In Section~\ref{sec:evp} we introduce the considered electromagnetic Stekloff eigenvalue problem and define the
associated holomorphic operator function $A_X(\cdot)$. We report in Theorem~\ref{thm:SpecBasic} that the spectrum of
$A_X(\cdot)$ is real and that $A_X(\lambda)$ is Fredholm if and only if $\lambda\neq0$.
In Section~\ref{sec:spectrumzero} we analyze the spectrum in a neighborhood of zero. We report in
Theorem~\ref{thm:SpecGapZero} that there exists $c_0>0$ so that $\sigma\big(A_X(\cdot)\big)\cap(0,c_0)=\emptyset$.
We report in Theorem~\ref{thm:existencelambdanegative} the existence of an infinite sequence of negative eigenvalues
which accumulate at zero.
In Section~\ref{sec:spectruminfty} we analyze the spectrum in a neighborhood of infinity. We report in
Theorem~\ref{thm:SpecGapInfty} that there exists $c_\infty>0$ so that $\sigma\big(A_X(\cdot)\big)\cap(-\infty,-c_\infty)
=\emptyset$. We report in Theorem~\ref{thm:existencelambdapositiv} the existence of an infinite sequence of positive
eigenvalues which accumulate at $+\infty$.
In Section~\ref{sec:conclusion} we collect our results in Proposition~\ref{prop:mainprop} and comment on the connection
between the original and the modified electromagnetic Stekloff eigenvalue problems.

\section{General setting}\label{sec:generalsetting}
In this section we set our notation and formulate assumptions on the domain and material parameters.
We also recall necessary results from different literature and adapt them to our setting.

\subsection{Functional analysis}
For generic Banach spaces $(X, \|\cdot\|_X)$, $(Y, \|\cdot\|_Y)$ denote $L(X,Y)$
the space of all bounded linear operators from $X$ to $Y$ with operator norm
$\|A\|_{L(X,Y)}:=\sup_{u\in X\setminus\{0\}} \|Au\|_Y/\|u\|_X$, $A\in L(X,Y)$.
We further set $L(X):=L(X,X)$. For generic Hilbert spaces
$(X, \spl\cdot,\cdot\spr_X)$, $(Y, \spl\cdot,\cdot\spr_Y)$ and $A\in L(X,Y)$ we denote $A^*\in L(Y,X)$ its adjoint
operator defined through $\spl u,A^*u' \spr_X=\spl Au,u'\spr_Y$ for all $u\in X,u'\in Y$.
Let $K(X,Y)\subset L(X,Y)$ be the space of compact operators and $K(X):=K(X,X)$.

We say that an operator $A\in L(X)$ is coercive, if $\inf_{u\in X\setminus\{0\}}
|\spl Au,u\spr_X|/\|u\|^2_X>0$. We say that $A\in L(X)$ is weakly coercive, if there
exists $K\in K(X)$ so that $A+K$ is coercive.
Let $\Lambda\subset\setC$ be open and consider an operator function $A(\cdot)\colon \Lambda\to L(X)$.
We call $A(\cdot)$ (weakly) coercive if $A(\lambda)$ is (weakly) coercive for all $\lambda\in\Lambda$.
We denote the spectrum of $A(\cdot)$ as
$\sigma\big(A(\cdot)\big):=\{\lambda\in\Lambda\colon A(\lambda)\text{ is not bijective}\}$ and
the resolvent set as $\rho\big(A(\cdot)\big):=\Lambda\setminus\sigma\big(A(\cdot)\big)$.
We denote $\sigma_\mathrm{ess}\big(A(\cdot)\big):=\{\lambda\in\Lambda\colon A(\lambda)\text{ is not Fredholm}\}$
the essential spectrum. For $A\in L(X)$ we set $\sigma(A):=\sigma\big(\cdot I-A\big)$,
$\sigma_\mathrm{ess}(A):=\sigma_\mathrm{ess}\big(\cdot I-A\big)$ and $\rho(A):=\rho\big(\cdot I-A\big)$.

\subsection{Lebesgue and Sobolev spaces}
Let $\Omega\subset\setR^3$ be a bounded path connected open Lipschitz domain and $\nv$ the outer unit
normal vector at $\partial\Omega$. Let $C^\infty_0(\Omega)$ be the space of infinitely many times differentiable
functions from $\Omega$ to $\setC$ with compact (closure of the) support in $\Omega$. We use standard notation for Lebesgue
and Sobolev spaces $L^2(\Omega)$, $L^\infty(\Omega)$, $W^{1,\infty}(\Omega)$, $H^s(\Omega)$ defined on the
domain $\Omega$ and $L^2(\partial\Omega)$, $H^s(\partial\Omega)$ defined on the
boundary $\partial\Omega$. We recall the continuity of the trace operator
$\tr \in L\big(H^s(\Omega), H^{s-1/2}\big)$ for all $s>1/2$.
For a vector space $X$ of scalar valued functions we
denote its bold symbol as space of three-vector valued functions
$\boldX:=X^3=X\times X\times X$, e.g.\ $\boldL^2(\Omega)$, $\boldH^s(\Omega)$,
$\boldL^2(\partial\Omega)$, $\boldH^s(\partial\Omega)$. For $\boldL^2(\partial\Omega)$
or a subspace, e.g.\ $\boldH^s(\partial\Omega), s>0$, the subscript $t$ denotes
the subspace of tangential fields. In particular $\boldL^2_t(\partial\Omega)=
\{u\in \boldL^2(\partial\Omega)\colon \nv\cdot u=0\}$ and $\boldH^s_t(\partial\Omega)=
\{u\in \boldH^s(\partial\Omega)\colon \nv\cdot u=0\}$. Let further $H^1_0(\Omega)$
be the subspace of $H^1(\Omega)$ of all functions with vanishing Dirichlet trace,
$H^1_*(\Omega)$ be the subspace of $H^1(\Omega)$ of all functions with vanishing mean,
i.e.\ $\spl u,1\spr_{L^2(\Omega)}=0$ and $H^1_*(\partial\Omega)$ be the subspace of
$H^1(\partial\Omega)$ of all functions with vanishing mean $\spl u,1 \spr_{L^2(\partial\Omega)}=0$.

\subsection{Additional function spaces}
Denote $\partial_{x_i} u$ the partial derivative of a function $u$ with respect
to the variable $x_i$. Let
\begin{align*}
\nabla u&:=(\partial_{x_1}u,\partial_{x_2}u,\partial_{x_3}u)^\top,\\
\div (u_1,u_2,u_3)^\top&:=\partial_{x_1}u_1+\partial_{x_2}u_2+\partial_{x_3}u_3,\\
\curl (u_1,u_2,u_3)^\top&:=(-\partial_{x_2}u_1+\partial_{x_1}u_3,
\partial_{x_3}u_1-\partial_{x_1}u_3,-\partial_{x_2}u_1+\partial_{x_1}u_2)^\top.
\end{align*}
For $\epsilon\in\big(L^\infty(\Omega)\big)^{3x3}$ let $\diveps u:=\div(\epsilon u)$.
For a bounded Lipschitz domain $\Omega$ let $\nabla_\partial, \div_\partial$ and $\curl_\partial=\nu\times\nabla_\partial$
be the respective differential operators for functions defined on $\partial\Omega$.
We recall that for $u\in\boldL^2(\Omega)$ with $\curl u\in\boldL^2(\Omega)$ the
tangential trace $\tr_{\nv\times}u\in \boldH^{-1/2}(\div_\partial;\partial\Omega):=\{u \in
\boldH^{-1/2}(\partial\Omega)\colon \div_\partial u\in H^{-1/2}(\partial\Omega)\}$,
$\|u\|^2_{\boldH^{-1/2}(\div_\partial;\partial\Omega)}:=\|u\|^2_{\boldH^{-1/2}(\partial\Omega)}
+\|\div_\partial u\|^2_{H^{-1/2}(\partial\Omega)}$ is well defined and
$\|\tr_{\nv\times}u\|_{\boldH^{-1/2}(\div_\partial;\partial\Omega)}^2$
is bounded by a constant times $\|u\|^2_{\boldL^2(\Omega)}+\|\curl u\|^2_{\boldL^2(\Omega)}$.
Likewise for $u\in\boldL^2(\Omega)$ with $\div u\in L^2(\Omega)$
the normal trace $\tr_{\nv\cdot}u\in H^{-1/2}(\partial\Omega)$ is well defined
and $\|\tr_{\nv\cdot}u\|_{H^{-1/2}(\partial\Omega)}^2$ is
bounded by a constant times $\|u\|_{\boldL^2(\Omega)}^2+\|\div u\|_{L^2(\Omega)}^2$.
Likewise for $u\in\boldL^2(\Omega)$ with $\diveps u\in L^2(\Omega)$
the normal trace $\tr_{\nv\cdot\epsilon}u\in H^{-1/2}(\partial\Omega)$ is well defined and
$\|\tr_{\nv\cdot\epsilon}u\|_{H^{-1/2}(\partial\Omega)}^2$ is bounded by a constant times
$\|\epsilon u\|_{\boldL^2(\Omega)}^2+\|\diveps u\|_{L^2(\Omega)}^2$.
For $\dd\in\{\curl,\div,\diveps,\tr_{\nv\times},\tr_{\nv\cdot},\tr_{\nv\cdot\epsilon}\}$ let
\begin{subequations}
\begin{align}
L^2(\dd):=\left\{\begin{array}{ll}
\boldL^2(\Omega),&\dd=\curl,\\
L^2(\Omega),&\dd=\div,\diveps,\\
\boldL^2_t(\partial\Omega),&\dd=\tr_{\nv\times},\\
L^2(\partial\Omega),&\dd=\tr_{\nv\cdot},\tr_{\nv\cdot\epsilon}
\end{array}\right..
\end{align}
Let
\begin{align}
\begin{aligned}
H(\dd;\Omega)&:=\{u\in L^2(\Omega)\colon \dd u\in L^2(\dd)\},\\
\spl u,u'\spr_{H(\dd;\Omega)}&:=\spl u,u'\spr_{L^2(\Omega)}
+\spl \dd u,\dd u'\spr_{L^2(\dd)},
\end{aligned}
\end{align}
\begin{align}
H(\dd^0;\Omega)&:=\{u\in H(\dd;\Omega)\colon \dd u=0\}.
\end{align}
Also for
\begin{align*}
\dd_1,\dd_2,\dd_3,\dd_4\in
\{&\curl,\div,\diveps,\tr_{\nv\times},\tr_{\nv\cdot},\tr_{\nv\cdot\epsilon},\\
&\curl^0,\div^0,\diveps^0,\tr_{\nv\times}^0,\tr_{\nv\cdot}^0,\tr_{\nv\cdot\epsilon}^0\}
\end{align*}
let
\begin{align}
\begin{aligned}
H(\dd_1,\dd_2;\Omega)&:=H(\dd_1;\Omega)\cap H(\dd_2;\Omega),\\
\spl u,u'\spr_{H(\dd_1,\dd_2;\Omega)}&:=\spl u,u'\spr_{L^2(\Omega)}
+\spl \dd_1 u,\dd_1 u'\spr_{L^2(\dd_1)}+\spl \dd_2 u,\dd_2 u'\spr_{L^2(\dd_2)},
\end{aligned}
\end{align}
\begin{align}
\begin{aligned}
H(\dd_1,\dd_2,\dd_3;\Omega)&:=H(\dd_1;\Omega)\cap H(\dd_2;\Omega)\cap H(\dd_3;\Omega),\\
\spl u,u'\spr_{H(\dd_1,\dd_2,\dd_3;\Omega)}&:=\spl u,u'\spr_{L^2(\Omega)}
+\spl \dd_1 u,\dd_1 u'\spr_{L^2(\dd_1)}+\spl \dd_2 u,\dd_2 u'\spr_{L^2(\dd_2)}\\
&+\spl \dd_3 u,\dd_3 u'\spr_{L^2(\dd_3)},
\end{aligned}
\end{align}
and
\begin{align}
\begin{aligned}
H(\dd_1,\dd_2,\dd_3,\dd_4;\Omega)&:=H(\dd_1;\Omega)\cap H(\dd_2;\Omega)
\cap H(\dd_3;\Omega)\cap H(\dd_4;\Omega),\\
\spl u,u'\spr_{H(\dd_1,\dd_2,\dd_3,\dd_4;\Omega)}&:=\spl u,u'\spr_{L^2(\Omega)}
+\spl \dd_1 u,\dd_1 u'\spr_{L^2(\dd_1)}+\spl \dd_2 u,\dd_2 u'\spr_{L^2(\dd_2)}\\
&+\spl \dd_3 u,\dd_3 u'\spr_{L^2(\dd_3)}+\spl \dd_4 u,\dd_4 u'\spr_{L^2(\dd_4)}.
\end{aligned}
\end{align}
\end{subequations}

\subsection{Assumption on the domain and material parameters}
\begin{assumption}[Assumption on $\epsilon$]\label{ass:eps}
Let $\epsilon\in \big(L^\infty(\Omega)\big)^{3x3}$ be a real, symmetric matrix function so that there exists
$c_\epsilon>0$ with
\begin{subequations}
\begin{align}
c_\epsilon |\xi|^2 &\leq \xi^H \epsilon(x) \xi
\end{align}
\end{subequations}
for all $x\in\Omega$ and all $\xi\in\setC^3$.
We further assume that there exists a Lipschitz domain $\hat\Omega\subset\Omega$ so that the closure of
$\hat\Omega$ is compact in $\Omega$ and $\epsilon|_{\Omega\setminus\hat\Omega}$ equals the
identity matrix $\I_{3\times3}\in\setC^{3\times3}$. 
\end{assumption}
We note that a generalization of Assumption~\ref{ass:eps} to $\epsilon|_{\Omega\setminus\hat\Omega}\in
W^{1,\infty}(\Omega\setminus\hat\Omega)$ seems possible. Let $\check\Omega\subset\Omega$ be a Lipschitz domain
so that the closure of $\check\Omega$ is compact in $\Omega$ and the closure of $\hat\Omega\subset\check\Omega$ is
compact in $\check\Omega$.
Let $\chi$ be infinitely many times differentiable, so that $\chi|_{\Omega\setminus\check\Omega}=1$ and
$\chi|_{\hat\Omega}=0$.

\begin{assumption}[Assumption on $\mu$]\label{ass:mu}
Let $\mu^{-1}\in \big(L^\infty(\Omega)\big)^{3x3}$ be a real, symmetric matrix function so that there exists
$c_\mu>0$ with
\begin{subequations}
\begin{align}
c_\mu |\xi|^2 &\leq \xi^H \mu^{-1}(x) \xi
\end{align}
\end{subequations}
for all $x\in\Omega$ and all $\xi\in\setC^3$.
We further assume that $\mu|_{\Omega\setminus\hat\Omega}$ equals the identity matrix $\I_{3\times3}\in\setC^{3\times3}$.
\end{assumption}

\begin{assumption}[Assumption on $\Omega$]\label{ass:Domain}
Let $\Omega\subset\setR^3$ be a bounded path connected Lipschitz domain so that there
exists $\delta>0$ and the following shift theorem holds on $\Omega$:
Let $f\in L^2(\Omega)$, $g\in H^{1/2}(\partial\Omega)$ with
$\spl g,1\spr_{L^2(\partial\Omega)}=0$ and $w\in H^1_*(\Omega)$ be the solution to
\begin{subequations}
\begin{align}
-\Delta w &= f\quad\text{ in }\Omega,\\
n\cdot\nabla w&= g\quad\text{ at }\partial\Omega.
\end{align}
\end{subequations}
Then the linear map $(f,g)\mapsto w\colon L^2(\Omega)\times
H^{1/2}(\partial\Omega)\to H^{3/2+\delta}(\Omega)$ is well defined and continuous.
\end{assumption}
The above assumption holds e.g.\ for smooth domains and Lipschitz polyhedral~\cite[Corollary~23.5]{Dauge:88}.

\begin{assumption}[Assumption on $\Omega, \epsilon$ and $\mu^{-1}$]\label{ass:UCP}
Let $\epsilon, \mu^{-1}$ and $\Omega$ be so that a unique continuation principle holds, i.e.\ if $u\in H(\curl;\Omega)$
solves
\begin{subequations}
\begin{align}
\curl\mu^{-1}\curl u -\omega^2\epsilon u&=0\quad\text{in }\Omega,\\
\tr_{\nv\times} u&=0\quad\text{at }\partial\Omega,\\
\tr_{\nv\times} \mu^{-1}\curl u&=0\quad\text{at }\partial\Omega,
\end{align}
\end{subequations}
then $u=0$.
\end{assumption}
To our knowledge the most general todays available result on the unique continuation principle for Maxwells equations
is the one of Ball, Capdeboscq and Tsering-Xiao~\cite{BallCapdeboscqTsering-Xiao:12}. It essentially requires $\epsilon$
and $\mu^{-1}$ to be piece-wise $W^{1,\infty}$.

\subsection{Trace regularities and compact embeddings}
We recall a classical result from Costabel~\cite{Costabel:90}:
\begin{subequations}\label{eq:CostabelTrace}
\begin{align}
\tr_{\nv\cdot}\in  L\big(H(\curl,\div,\tr_{\nv\times};\Omega), L^2(\partial\Omega)\big),\\
\tr_{\nv\times}\in L\big(H(\curl,\div,\tr_{\nv\cdot};\Omega), \boldL^2_t(\partial\Omega)\big),
\end{align}
\end{subequations}
and
\begin{align}\label{eq:CostabelDomain}
\begin{aligned}
&\text{The embeddings from }H(\curl,\div,\tr_{\nv\cdot};\Omega)\text{ and }\\
&H(\curl,\div,\tr_{\nv\times};\Omega)\text{ to } \boldH^{1/2}(\Omega)\text{ are bounded}.
\end{aligned}
\end{align}
We adapt the trace results of Costabel to our setting in the next lemmata.

\begin{lemma}\label{lem:ntrequal}
Let $\epsilon$ suffice Assumption~\ref{ass:eps}. Thence
\begin{align*}
\tr_{\nv\cdot} \in L\big(H(\diveps;\Omega), H^{-1/2}(\partial\Omega)\big) 
\quad\text{and}\quad \tr_{\nv\cdot}=\tr_{\nv\cdot\epsilon}.
\end{align*}
and
\begin{align*}
\tr_{\nv\cdot\epsilon} \in L\big(H(\div;\Omega), H^{-1/2}(\partial\Omega)\big)
\quad\text{and}\quad \tr_{\nv\cdot\epsilon}=\tr_{\nv\cdot}.
\end{align*}
\end{lemma}
\begin{proof}
If $u\in H(\diveps;\Omega)$ then $\chi u\in H(\div;\Omega)$. Since $\chi|_{\Omega\setminus\check\Omega}=
\epsilon|_{\Omega\setminus\check\Omega}=1$ it follows $\tr_{\nv\times} u=\tr_{\nv\times} \chi u
=\tr_{\nv\times} \epsilon \chi u$. The reverse direction follows the same way.
\end{proof}

\begin{lemma}\label{lem:trLtwo}
Let $\epsilon$ suffice Assumption~\ref{ass:eps}. Thence
\begin{subequations}\label{eq:CostabelTraceEps}
\begin{align}
\label{eq:CostabelTraceEpsA}
\tr_{\nv\cdot\epsilon}\in  L\big(H(\curl,\diveps,\tr_{\nv\times};\Omega), L^2(\partial\Omega)\big),\\
\label{eq:CostabelTraceEpsB}
\tr_{\nv\times}\in L\big(H(\curl,\diveps,\tr_{\nv\cdot\epsilon};\Omega), \boldL^2_t(\partial\Omega)\big),
\end{align}
\end{subequations}
\end{lemma}
\begin{proof}
Apply \eqref{eq:CostabelTrace} to $\chi u$ and employ Lemma~\ref{lem:ntrequal}.
\end{proof}

We deduce the next lemma from Amrouche, Bernardi, Dauge and Girault~\cite{AmroucheBernardiDaugeGirault:98}.
\begin{lemma}\label{lem:Vtraceregularity}
Let $\epsilon$ suffice Assumption~\ref{ass:eps} and $\Omega$ suffice Assumption~\ref{ass:Domain}. Thence
\begin{align}
\tr_{\nv\times} \in L\big( H(\curl,\diveps,\tr_{\nv\cdot\epsilon}^0;\Omega), \boldH^\delta_t(\partial\Omega) \big).
\end{align}
In particular $\tr_{\nv\times} \in L\big( H(\curl,\diveps,\tr_{\nv\cdot\epsilon}^0;\Omega),
\boldL^2_t(\partial\Omega) \big)$ is compact.
\end{lemma}
\begin{proof}
Apply the proof of~\cite[Proposition~3.7]{AmroucheBernardiDaugeGirault:98} to $\chi u$ and employ
Assumption~\ref{ass:Domain} to obtain $\chi u\in\boldH^{1/2+\delta}(\Omega)$. Employ
$\tr\in L\big(\boldH^{1/2+\delta}(\Omega), \boldH^{\delta}(\partial\Omega)\big)$ and the compact embedding
$\boldH^\delta_t(\partial\Omega)\to \boldL^2_t(\partial\Omega)$.
\end{proof}

We recall from Weber~\cite{Weber:80}:
\begin{align}\label{eq:Weber}
\begin{aligned}
&\text{The embeddings from }H(\curl,\diveps,\tr_{\nv\cdot\epsilon}^0;\Omega)\text{ and}\\
&H(\curl,\diveps,\tr_{\nv\times}^0;\Omega)\text{ to }\boldL^2(\Omega)\text{ are compact},
\end{aligned}
\end{align}
if $\epsilon$ suffices Assumption~\ref{ass:eps}. We mention that Weber~\cite{Weber:80} presumes $\Omega$ to have the
cone property, which is however equivalent to the Lipschitz property~\cite[Theorem~1.2.2.2]{Grisvard:85}.
\begin{lemma}\label{lem:compactEmbedding}
Let $\epsilon$ suffice Assumption~\ref{ass:eps}. Thence the embedding
\begin{align*}
H(\curl,\diveps,\tr_{\nv\times};\Omega) \to \boldL^2(\Omega)
\end{align*}
is compact.
\end{lemma}
\begin{proof}
Let $E: H(\curl,\diveps,\tr_{\nv\cdot\epsilon};\Omega)
\to \boldL^2(\Omega)\colon u\mapsto u$. Let $M(\alpha)$ be the multiplication operator
with symbol $\alpha$.
We split the identity operator in two parts $I=M(\chi)+M(1-\chi)$. Thence
$EM(\chi)$ is compact due to \eqref{eq:CostabelDomain} and $EM(1-\chi)$ is
compact due to \eqref{eq:Weber}. Hence $E=EM(\chi)+EM(1-\chi)$ is compact too.
\end{proof}

\subsection{Helmholtz decomposition on the boundary}
We recall from Buffa, Costabel and Sheen~\cite[Theorem 5.5]{BuffaCostabelSheen:02}:
\begin{align}
\boldL^2_t(\partial\Omega)=\nabla_\partial H^1(\partial\Omega) \oplus^\bot
\curl_\partial H^1(\partial\Omega).
\end{align}
and denote the respective orthogonal projections by
\begin{align}
P_{\nabla_\partial}\colon \boldL^2_t(\partial\Omega) \to \nabla_\partial H^1(\partial\Omega),\qquad
P_{\nabla_\partial^\top}\colon \boldL^2_t(\partial\Omega) \to \curl_\partial H^1(\partial\Omega).
\end{align}
Recall $\div_\partial \tr_{\nv\times} \in L\big(H(\curl;\Omega), H^{-1/2}(\partial\Omega)\big)$.
So for $u\in H(\curl;\Omega)$ let $z$ be the solution to find $z\in H^1_*(\partial\Omega)$ so that
\begin{align}
\spl \nabla_\partial z, \nabla_\partial z' \spr_{\boldL_t^2(\partial\Omega)}
= -\spl \div_\partial \tr_{\nv\times} u, z' \spr_{H^{-1}(\partial\Omega)\times H^1(\partial\Omega)}
\end{align}
for all $z'\in H^1_*(\partial\Omega)$ and set
\begin{align}\label{eq:DefS}
Su:=\nabla_\partial z.
\end{align}
From the construction of $S$ it follows $S\in L\big(H(\curl;\Omega),\boldL_t^2(\partial\Omega)\big)$ and further
\begin{align}
Su=P_{\nabla_\partial}\tr_{\nv\times} u
\end{align}
for $u\in H(\curl,\tr_{\nv\times};\Omega)$.

\section{The electromagnetic Stekloff eigenvalue problem}\label{sec:evp}
Let $\omega>0$ be fixed. For $\lambda\in\setC$ let $A(\lambda)\in L\big(H(\curl,\tr_{\nv\times};\Omega)\big)$ be defined
through
\begin{align}
\begin{aligned}
\spl A(\lambda)u,u'&\spr_{H(\curl,\tr_{\nv\times};\Omega)}:=
\spl \mu^{-1} \curl u,\curl u'\spr_{\boldL^2(\Omega)}
-\omega^2\spl\epsilon u,u'\spr_{\boldL^2(\Omega)}\\
&-\lambda \spl \tr_{\nv\times} u,\tr_{\nv\times} u'\spr_{\boldL^2_t(\partial\Omega)}
\quad\text{for all }u,u'\in H(\curl,\tr_{\nv\times};\Omega).
\end{aligned}
\end{align}
The electromagnetic Stekloff eigenvalue problem which we investigate in this note is to
\begin{align}\label{eq:EVP}
\text{find}\quad (\lambda,u)\in \setC\times H(\curl,\tr_{\nv\times};\Omega)
\setminus\{0\}\quad\text{so that}\quad A(\lambda)u=0.
\end{align}
We note that the sign of $\lambda$ herein is reversed compared to~\cite{CamanoLacknerMonk:17}. Let
\begin{align}
\spl u,u' \spr_{\tilde X}:=\spl \mu^{-1}\curl u,\curl u' \spr_{\boldL^2(\Omega)}
+\spl \epsilon u,u' \spr_{\boldL^2(\Omega)}
+\spl \tr_{\nv\times} u, \tr_{\nv\times} u' \spr_{\boldL_t^2(\partial\Omega)}
\end{align}
for all $u,u'\in H(\curl,\tr_{\nv\times};\Omega)$. It is straight forward to see that the norms induced by
$\spl \cdot,\cdot \spr_{\tilde X}$ and $\spl \cdot,\cdot \spr_{H(\curl,\tr_{\nv\times};\Omega)}$ are equivalent.
To analyze the operator $A(\lambda)$ we introduce the following subspaces of $H(\curl,\tr_{\nv\times};\Omega)$:
\begin{subequations}
\begin{align}
V&:=H(\curl,\diveps^0,\tr_{\nv\times},\tr_{\nv\cdot\epsilon}^0;\Omega),\\
W_1&:=H(\curl^0,\diveps^0,\tr_{\nv\times};\Omega) \cap W_2^{\bot_{\tilde X}},\\
W_2&:=H(\curl^0,\tr_{\nv\times}^0;\Omega).
\end{align}
\end{subequations}
We recall \cite[Theorem~4.3 and Remark~4.4]{Monk:03}:
\begin{align}
\begin{aligned}
K_N(\Omega):=\{&\nabla u\colon u\in H^1(\Omega), \diveps u=0\text{ in }\Omega,\\
&\tr u\text{ is constant on each of the connected parts of }\partial\Omega\}
\end{aligned}
\end{align}
and $\dim K_N(\Omega)=\text{number of connected parts of }\partial\Omega-1<\infty$. It holds
\begin{align}\label{eq:W2char}
W_2=\nabla H^1_0(\Omega)\oplus^{\bot_{\tilde X}} K_N(\Omega).
\end{align}
Thus
\begin{align}
\begin{aligned}\label{eq:W1char}
W_1=\{&\nabla u\colon u\in H^1(\Omega), \quad \diveps u=0 \text{ in }\Omega,
\quad \tr_{\nv\cdot\epsilon}\nabla u\in L^2(\partial\Omega),
\\&\spl \tr_{\nv\cdot\epsilon}\nabla u, 1\spr_{L^2(\Gamma)}=0\quad \text{for each }\Gamma
\text{ of the connected parts of }\partial\Omega\}.
\end{aligned}
\end{align}
We continue with a decomposition of $H(\curl,\tr_{\nv\times};\Omega)$, which is similar but different to
\cite[Theorem~3.1]{Halla:19StekloffAppr}.

\begin{theorem}\label{thm:VW}
Let $\epsilon$ suffice Assumption~\ref{ass:eps} and $\mu$ suffice Assumption~\ref{ass:mu}. Thence
\begin{align}
H(\curl,\tr_{\nv\times};\Omega)=(V\oplus W_1)\oplus^{\bot_{\tilde X}} W_2
\end{align}
in the following sense. There exist projections $P_V, P_{W_1}, P_{W_2} \in L\big(H(\curl,\tr_{\nv\times};\Omega)\big)$
with $\ran P_V=V, \ran P_{W_1}=W_1, \ran P_{W_2}=W_2$,
$W_1, W_2 \subset \ker P_V$, $V, W_2 \subset \ker P_{W_1}$, $V, W_1 \subset \ker P_{W_2}$
and $u=P_vu+P_{W_1}u+P_{W_2}u$ for each $u\in H(\curl,\tr_{\nv\times};\Omega)$.
Thus, the norm induced by
\begin{align}\label{eq:scptildeX}
\begin{aligned}
\spl u,u'\spr_X&:=\spl P_Vu,P_Vu'\spr_{\tilde X}
+\spl P_{W_1}u,P_{W_1}u'\spr_{\tilde X}+\spl P_{W_2}u,P_{W_2}u'\spr_{\tilde X},
\end{aligned}
\end{align}
$u,u'\in H(\curl,\tr_{\nv\times};\Omega)$, is equivalent to $\|\cdot\|_{H(\curl,\tr_{\nv\times};\Omega)}$.
\end{theorem}
\begin{proof}
\textit{1.\ Step:}\quad
Let $P_{W_2}$ be the $\tilde X$-orthogonal projection onto $W_2$.
Hence $P_{W_2}\in L\big(H(\curl,\tr_{\nv\times};\Omega)\big)$ is a projection with range $W_2$ and kernel
\begin{align*}
W_2^{\bot_{H(\curl,\tr_{\nv\times};\Omega)}}\supset V, W_1.
\end{align*}

\textit{2a.\ Step:}\quad
Let $u\in H(\curl,\tr_{\nv\times};\Omega)$. Note that due to $\diveps (u-P_{W_2}u)=0$ and Lemma~\ref{lem:trLtwo} it hold
$\tr_{\nv\cdot\epsilon}(u-P_{W_2}u) \in L^2(\partial\Omega)$ and
$\spl \tr_{\nv\cdot\epsilon}(u-P_{W_2}u), 1\spr_{L^2(\Gamma)}=0$ for each $\Gamma$ of the connected parts of
$\partial\Omega$. Let $w_*\in H^1_*(\Omega)$ be the unique solution to
\begin{align*}
-\diveps \nabla w_*=0 \quad\text{in }\Omega, \qquad \nv\cdot \epsilon\nabla w_*= \tr_{\nv\cdot\epsilon}(u-P_{W_2}u)
\quad\text{at }\partial\Omega.
\end{align*}
Let $P_{W_1}u:=\nabla w_*$. By construction of $P_{W_1}$ and due to Lemma~\ref{lem:trLtwo} it hold
$P_{W_1}\in L\big(H(\curl,\tr_{\nv\times};\Omega)\big)$ and $\ran P_{W_1}\subset W_1$.
Let $u\in W_1$. Then $P_{W_2}u=0$ and hence $P_{W_1}u=u$. Thus $P_{W_1}$ is a projection and $\ran P_{W_1}=W_1$.

\textit{2b.\ Step:}\quad
If $u\in W_2$ then $u-P_{W_2}u=0$, further $\tr_{\nv\cdot\epsilon}(u-P_{W_2}u)=0$ and thus $P_{W_1}u=0$.
Thus $W_2\subset \ker P_{W_1}$.
If $u\in V$ then $P_{W_2}u=0$, further $\tr_{\nv\cdot\epsilon}(u-P_{W_2}u)=\tr_{\nv\cdot\epsilon}u=0$ and thus
$P_{W_1}u=0$. Hence $V\subset \ker P_{W_1}$.

\textit{3.\ Step:}\quad
Let $u\in H(\curl,\tr_{\nv\times};\Omega)$ and $P_Vu:=u-P_{W_1}u-P_{W_2}u$. It follow
$P_V\in L\big(H(\curl,\tr_{\nv\times};\Omega)\big)$, $P_Vu\in V$ and $P_VP_Vu=P_Vu$.
If $u\in V$ then $P_Vu=u$ and hence $\ran P_V=V$. It follow further $W_1,W_2 \subset \ker P_V$.

\textit{4.\ Step:}\quad
By means of the triangle
inequality and a Young inequality it holds.
\begin{align*}
\|u\|_{\tilde X}^2 =\|P_Vu+P_{W_1}u+P_{W_2}u\|_{\tilde X}^2
&\leq 3\big(\|P_Vu\|_{\tilde X}^2 +\|P_{W_1}u\|_{\tilde X}^2 +\|P_{W_2}u\|_{\tilde X}^2\big)\\
&=3\|u\|_X^2.
\end{align*}
On the other hand due to the boundedness of the projections
\begin{align*}
\|u\|_X^2
&=\|P_Vu\|_{{\tilde X}}^2
+\|P_{W_1}u\|_{{\tilde X}}^2
+\|P_{W_2}u\|_{{\tilde X}}^2\\
&\leq \big(\|P_V\|_{L(\tilde X)}^2
+\|P_{W_1}\|_{L(\tilde X)}^2
+\|P_{W_2}\|_{L(\tilde X)}^2\big)
\|u\|_{{\tilde X}}^2.
\end{align*}
Thus $\|\cdot\|_X$ is equivalent to $\|\cdot\|_{\tilde X}$. Since $\|\cdot\|_{\tilde X}$ is equivalent to
$\|\cdot\|_{H(\curl,\tr_{\nv\times};\Omega)}$, $\|\cdot\|_X$ is also equivalent to
$\|\cdot\|_{H(\curl,\tr_{\nv\times};\Omega)}$.
\end{proof}

Let us look at $A(\lambda)$ in light of this substructure of ${H(\curl,\tr_{\nv\times};\Omega)}$. To this end we
consider the space
\begin{align}\label{eq:X}
X:=H(\curl,\tr_{\nv\times};\Omega), \qquad
\spl\cdot,\cdot\spr_X\quad\text{as defined in~\eqref{eq:scptildeX}}.
\end{align}
It follows that $P_V, P_{W_1}$ and $P_{W_1}$ are even orthogonal projections in $X$. Let further $A_X(\cdot), A_c,
A_\epsilon,  A_{\tr} \in L(X)$ be defined through
\begin{subequations}
\begin{align}
\label{eq:DefAX}
\spl A_X(\lambda) u,u'\spr_X&:=\spl A(\lambda)u,u'\spr_{H(\curl,\tr_{\nv\times};\Omega)}
\quad\text{for all }u,u'\in X,\lambda\in\setC\\
\spl A_c u,u'\spr_X&:=\spl \mu^{-1}\curl u,\curl u'\spr_{\boldL^2(\Omega)}
\quad\text{for all }u,u'\in X,\\
\spl A_\epsilon u,u'\spr_X&:=\spl \epsilon u,u'\spr_{\boldL^2(\Omega)}
\quad\text{for all }u,u'\in X,\\
\spl A_{\tr} u,u'\spr_X&:=\spl \tr_{\nv\times}u,\tr_{\nv\times}u'\spr_{\boldL^2_t(\partial\Omega)}
\quad\text{for all }u,u'\in X.
\end{align}
\end{subequations}
From the definitions of $V, W_1$ and $W_2$ we deduce that
\begin{align}\label{eq:AVWW}
\begin{aligned}
A_X(\lambda)&=
(P_V+P_{W_1}+P_{W_2})(A_c-\omega^2 A_\epsilon-\lambda A_{\tr})(P_V+P_{W_1}+P_{W_2})\\
&=P_VA_cP_V
-\omega^2\big(P_VA_\epsilon P_V
+P_{W_1}A_\epsilon P_{W_1}
+P_{W_2}A_\epsilon P_{W_2}\big)\\
&-\lambda (P_V+P_{W_1})A_{\tr} (P_V+P_{W_1})\\
&=P_VA_cP_V
-\omega^2(P_VA_\epsilon P_V
+P_{W_1}A_\epsilon P_{W_1}
+P_{W_2}A_\epsilon P_{W_2}\big)\\
&-\lambda \big(P_VA_{\tr} P_V
+P_{W_1}A_{\tr} P_{W_1}
+P_VA_{\tr} P_{W_1}
+P_{W_1}A_{\tr} P_V\big).
\end{aligned}
\end{align}
If we identify $X\sim V\times W_1\times W_2$ and $X \ni u\sim (v,w_1,w_2)\in V\times W_1\times W_2$, we can
identify $A_X(\lambda)$ with the block operator
\begin{align}\label{eq:Ablock}
\bpm
P_V(A_c -\omega^2 A_\epsilon -\lambda A_{\tr}) |_V &-\lambda P_V A_{\tr} |_{W_1}&\\
-\lambda P_{W_1} A_{\tr} |_V&-P_{W_1}(\omega^2 A_\epsilon+\lambda A_{\tr}) |_{W_1}&\\
&&-\omega^2 P_{W_2} A_\epsilon |_{W_2}
\epm.
\end{align}

\begin{theorem}\label{thm:SpecBasic}
Let $\epsilon$ suffice Assumption~\ref{ass:eps}, $\mu$ suffice Assumption~\ref{ass:mu} and $\Omega$ suffice
Assumption~\ref{ass:Domain}. Thence $A_X(\lambda)$ is Fredholm if and only if $\lambda\in\setC\setminus\{0\}$.

If in addition Assumption~\ref{ass:UCP} holds true, then $\sigma\big(A(\cdot)\big)\subset\setR$ and
$\sigma\big(A(\cdot)\big)\setminus\{0\}$ consists of an at most countable set of eigenvalues with finite algebraic
multiplicity which have no accumulation point in $\setR\setminus\{0\}$.
\end{theorem}
\begin{proof}
The first statement follows from Theorem~3.2 and Corollary~3.4 of \cite{Halla:19StekloffAppr}.
The second statement can be seen as in the proof of Corollary~3.3 of \cite{Halla:19StekloffAppr}.
\end{proof}
From~\eqref{eq:AVWW} or \eqref{eq:Ablock} we recognize that any eigenfunction $u\in X$ satisfies $P_{W_2}u=w_2=0$.
Hence to study the eigenvalues of $A_X(\cdot)$ it suffices to study
\begin{align}
\begin{aligned}
(P_V+P_{W_1})A_X&(\lambda)|_{V\oplus W_1} \\&\sim
\bpm
P_V(A_c -\omega^2 A_\epsilon -\lambda A_{\tr}) |_V &-\lambda P_V A_{\tr} |_{W_1}\\
-\lambda P_{W_1} A_{\tr} |_V&-P_{W_1}(\omega^2 A_\epsilon+\lambda A_{\tr}) |_{W_1}
\epm.
\end{aligned}
\end{align}

\section{Spectrum in the neighborhood of zero}\label{sec:spectrumzero}
First, we establish in Theorem~\ref{thm:SpecGapZero} the absence of eigenvalues of $A_X(\cdot)$ in $(0,c)$ for
sufficiently small $c>0$. Later on in Theorem~\ref{thm:existencelambdanegative}, we establish the existence of an
infinite sequence of negative eigenvalues of $A_X(\cdot)$ which accumulate at zero.

\subsection{Spectrum right of zero}
We will require in this section the following additional assumption.
\begin{assumption}[$\omega^2$ is no Neumann eigenvalue]\label{ass:NoNeumann}
\begin{align*}
P_{V} A_c |_{V}  -\omega^2 P_{V} A_\epsilon |_{V} \in L(V) \qquad\text{is bijective}.
\end{align*}
\end{assumption}
Due to Assumption~\ref{ass:NoNeumann} we know that $P_V (A_c -\omega^2 A_\epsilon) |_V$ is invertible.
Thus by a Neumann series argument $P_V (A_c -\omega^2 A_\epsilon-\lambda A_{\tr}) |_V \in L(V)$ is invertible too for all
\begin{align}\label{eq:lambdasmalli}
|\lambda| < \frac{1}{\|(P_V (A_c -\omega^2 A_\epsilon) |_V)^{-1} P_V A_{\tr} |_V\|_{L(V)}}
\end{align}
and thence it holds
\begin{align}
\|(P_V (A_c -\omega^2 A_\epsilon-\lambda A_{\tr}) |_V)^{-1}\|_{L(V)}
\leq  \frac{1}{1-\lambda \|(P_V (A_c -\omega^2 A_\epsilon) |_V)^{-1} P_V A_{\tr} |_V\|_{L(V)}}.
\end{align}
For $\lambda$ satisfying~\eqref{eq:lambdasmalli} we build the Schur-complement of
$(P_V+P_{W_1})A_X(\lambda)|_{V\oplus W_1}$ with respect to $P_Vu=v$:
\begin{subequations}
\begin{align}
\label{eq:AW}
A_{W_1}(\lambda)&:=-\omega^2 P_{W_1} A_\epsilon |_{W_1} -\lambda(P_{W_1} A_{\tr} |_{W_1}+H_{W_1}(\lambda)) \in L(W_1),\\
H_{W_1}(\lambda)&:= \lambda P_{W_1} A_{\tr} (P_V (A_c -\omega^2 A_\epsilon-\lambda A_{\tr}) |_V)^{-1} P_V A_{\tr}|_{W_1}
\in L(W_1).
\end{align}
\end{subequations}
It is straight forward to see, that for $\lambda$ satisfying \eqref{eq:lambdasmalli}, $\lambda$ is an eigenvalue to
$A_X(\cdot)$ if and only if $\lambda$ is an eigenvalue to $A_{W_1}(\cdot)$. Hence to study the eigenvalues of $A_X(\cdot)$
in a neighborhood of zero, it completely suffices to study the eigenvalues of $A_{W_1}(\cdot)$ in a neighborhood of zero.
For
\begin{align}\label{eq:lambdasmallii}
|\lambda| < \frac{1}{2\|(P_V (A_c -\omega^2 A_\epsilon) |_V)^{-1} P_V A_{\tr} |_V\|_{L(V)}} 
\end{align}
we deduce
\begin{align}\label{eq:lambdasmalliii}
\|H_{W_1}(\lambda)\|_{L(W_1)} \leq \lambda 2 \|P_V\|_{L(X)} \|P_{W_1}\|_{L(X)} \|A_{\tr}\|_{L(X)}^2.
\end{align}
Let
\begin{align}
B_{\tr} \in L\big(X,\boldL^2_t(\partial\Omega)\big) \colon u\mapsto \tr_{\nv\times}u
\end{align}
so that
\begin{align}\label{eq:AeqBsB}
A_{\tr}=B_{\tr}^*B_{\tr}.
\end{align}

\begin{lemma}\label{lem:PWAtrisPD}
Let Assumptions~\ref{ass:eps} hold true. Thence $P_{W_1} A_{\tr} |_{W_1}$ is strictly positive definite, i.e.\
\begin{align}
\inf_{w_1\in W_1\setminus\{0\}} \frac{\spl (P_{W_1} A_{\tr} |_{W_1})w_1,w_1 \spr_X}{\|w_1\|_X^2}>0.
\end{align}
\end{lemma}
\begin{proof}
$A_{\tr}$ is selfadjoint and positive semi definite due to \eqref{eq:AeqBsB} and hence so is $P_{W_1} A_{\tr} |_{W_1}$.
$P_{W_1} A_{\tr} |_{W_1}$ is weakly coercive due to Lemma~\ref{lem:compactEmbedding} and $\curl w_1=0$ for each
$w_1\in W_1$.
$P_{W_1} A_{\tr} |_{W_1}$ is injective since $w_1\in W_1\cap \ker (P_{W_1} A_{\tr} |_{W_1})$ implies $w_1\in W_2$ and
hence $w_1=0$.
Since $P_{W_1} A_{\tr} |_{W_1}$ is selfadjoint, positive semi definite and bijective, it is already strictly positive
definite.
\end{proof}

\begin{lemma}\label{lem:PWAtrHisPD}
Let Assumptions~\ref{ass:eps}, \ref{ass:mu}, \ref{ass:Domain} and~\ref{ass:NoNeumann} hold
true. Thence there exists $c_0>0$ so that $P_{W_1} A_{\tr} |_{W_1}+H_{W_1}(\lambda)$ is strictly positive definite, i.e.\
\begin{align}
\inf_{w_1\in W_1\setminus\{0\}} \frac{\spl (P_{W_1} A_{\tr} |_{W_1}+H_{W_1}(\lambda))w_1,w_1 \spr_X}{\|w_1\|_X^2}>0,
\end{align}
for each $\lambda\in (-c_0,c_0)$.
\end{lemma}
\begin{proof}
It is straight forward to see that $H_{W_1}(\lambda)$ selfadjoint for $\lambda\in\setR$ satisfying \eqref{eq:lambdasmalli}.
The inverse triangle inequality, Lemma~\ref{lem:PWAtrisPD} and \eqref{eq:lambdasmallii}, \eqref{eq:lambdasmalliii} yield
the claim.
\end{proof}

\begin{theorem}\label{thm:SpecGapZero}
Let Assumptions~\ref{ass:eps}, \ref{ass:mu}, \ref{ass:Domain}, \ref{ass:UCP} and~\ref{ass:NoNeumann} hold
true and $c_0$ be as in Lemma~\ref{lem:PWAtrHisPD}. Thence $\sigma\big(A_X(\cdot)\big)\cap(0,c_0)=\emptyset$.
\end{theorem}
\begin{proof}
For $\lambda\in(0,c_0)$, we can build the Schur complement $A_{W_1}(\lambda)$ of $A_X(\lambda)$ with respect to
$P_Vu=v$ and $A_X(\lambda)$ is bijective if and only if $A_{W_1}(\lambda)$ is so. It follows from the definition
\eqref{eq:AW} of $A_{W_1}(\lambda)$ and Lemma \ref{lem:PWAtrHisPD} that $A_{W_1}(\lambda)$ is strictly positive definite for
$\lambda\in(0,c_0)$ and hence bijective.
\end{proof}

\subsection{Spectrum left of zero}
To study the eigenvalues of $A_{W_1}(\cdot)$ in $(-c_0,0)$ we introduce
\begin{align}
A_{W_1}(\tau,\lambda):=-\omega^2 P_{W_1} A_\epsilon |_{W_1} -\tau(P_{W_1} A_{\tr} |_{W_1}+H_{W_1}(\lambda)).
\end{align}
We notice that $\lambda\in(-c_0,0)$ is an eigenvalue of $A_{W_1}(\cdot)$, if and only if $\tau$ is an eigenvalue of
$A_{W_1}(\cdot,\lambda)$ and $\tau=\lambda$. We prove the existence of infinite eigenvalues of $A_{W_1}(\cdot)$ in
$(-c_0,0)$ by the fixed point technique outlined in~\cite{CakoniHaddar:09}.
\begin{lemma}\label{lem:existencetaun}
Let Assumptions~\ref{ass:eps}, \ref{ass:mu}, \ref{ass:Domain}, \ref{ass:UCP} and~\ref{ass:NoNeumann} hold true
and $c_0$ be as in Lemma~\ref{lem:PWAtrHisPD}.
Let $\lambda\in(-c_0,c_0)$. The spectrum of $A_{W_1}(\cdot,\lambda)$ consists of
$\sigma_\mathrm{ess}\big(A_{W_1}(\cdot,\lambda)\big)=\{0\}$ and an infinite sequence of negative eigenvalues
$(\tau_n(\lambda))_{n\in\setN}$ which accumulate at zero.
\end{lemma}
\begin{proof}
Due to Lemma~\ref{lem:PWAtrHisPD} $(P_{W_1} A_{\tr} |_{W_1}+H_{W_1}(\lambda))^{-1/2}$ is well defined and selfadjoint.
It holds $\dim W_1=\infty$ due to \eqref{eq:W1char}. The spectra of $A_{W_1}(\cdot,\lambda)$ and
\begin{align*}
&(P_{W_1} A_{\tr} |_{W_1}+H_{W_1}(\lambda))^{-1/2} A_{W_1}(\cdot,\lambda) (P_{W_1} A_{\tr} |_{W_1}+H_{W_1}(\lambda))^{-1/2}\\
= -\omega^2 &(P_{W_1} A_{\tr} |_{W_1}+H_{W_1}(\lambda))^{-1/2} P_{W_1} A_\epsilon |_{W_1} (P_{W_1} A_{\tr} |_{W_1}+H_{W_1}(\lambda))^{-1/2}
-\cdot I_{W_1}
\end{align*}
coincide. The latter  is the pencil of a standard eigenvalue problem for a compact selfadjoint non-positive injective
operator on an infinite dimensional Hilbert space and respective properties follow.
\end{proof}

\begin{lemma}\label{lem:continuouitytaun}
Let Assumptions~\ref{ass:eps}, \ref{ass:mu}, \ref{ass:Domain}, \ref{ass:UCP} and~\ref{ass:NoNeumann} hold true and $c_0$
be as in Lemma~\ref{lem:PWAtrHisPD}.
Let the sequence of negative eigenvalues $(\tau_n(\lambda))_{n\in\setN}$ to the operator function $A_{W_1}(\cdot,\lambda)$
be ordered non-decreasingly with multiplicity taken into account. The function
$(-c_0,c_0)\to\setR:\lambda\mapsto\tau_n(\lambda)$ is continuous for each $n\in\setN$.
\end{lemma}
\begin{proof}
Follows from the ordering of $(\tau_n(\lambda))_{n\in\setN}$ and \cite[\S~3]{Kato:95} or
\cite[Proposition~5.4]{SanchezSanchez:89}.
\end{proof}

\begin{theorem}\label{thm:existencelambdanegative}
Let Assumptions~\ref{ass:eps}, \ref{ass:mu}, \ref{ass:Domain}, \ref{ass:UCP} and~\ref{ass:NoNeumann} hold true. Thence
there exists an infinite sequence $(\lambda_n)_{n\in\setN}$ of negative eigenvalues to $A_X(\cdot)$ which accumulate at
zero.
\end{theorem}
\begin{proof}
Let $(\tau_n(\lambda))_{n\in\setN}$ be as in Lemma~\ref{lem:continuouitytaun}. Let $\lambda\in(-c_0,0)$.
Let $n_1\in\setN$ be so that $\lambda<\tau_{n_1}(\lambda)$. Consider the function $f_1(t):=\tau_{n_1}(t)-t$.
It hold: $f_1$ is continuous on $(-c_0,c_0)$ due to Lemma~\ref{lem:continuouitytaun}, $f_1(\lambda)>0$ and
$f_1(0)=\tau_{n_1}(0)<0$. It follows from the Intermediate Value Theorem that there exists $\lambda_1\in (\lambda,0)$
with $f_1(\lambda_1)=0$, i.e.\ $\lambda_1$ is an eigenvalue to $A_{W_1}(\cdot)$.

Let now $\lambda\in(\lambda_1,0)$ and $n_2\in\setN$ be so that $\lambda<\tau_{n_2}(\lambda)$. We can repeat the
former procedure to construct a second eigenvalue $\lambda_2\in(\lambda_1,0)$ to $A_{W_1}(\cdot)$. Since
$\lambda_2\in(\lambda_1,0)$, $\lambda_2$ is distinct from $\lambda_1$.
We can repeat the former procedure inductively to construct a sequence  $(\lambda_n\in(-c_0,0))_{n\in\setN}$ of pairwise
distinct eigenvalues to $A_{W_1}(\cdot)$.

As already discussed, the spectra of $A_{W_1}(\cdot)$ and $A_X(\cdot)$ coincide on the ball \eqref{eq:lambdasmalli}.
Since $[-c_0,0]$ is compact and the sequence $(\lambda_n\in(-c_0,0))_{n\in\setN}$ has an infinite index set,
$(\lambda_n)_{n\in\setN}$ admits a cluster point in $[-c_0,0]$.
Due to Theorem~\ref{thm:SpecBasic} $\sigma\big(A_X(\cdot)\big)$ admits no cluster points in $\setC\setminus\{0\}$. Thus
$(\lambda_n)_{n\in\setN}$ accumulate at zero. The claim is proven.
\end{proof}

\section{Spectrum in the neighborhood of infinity}\label{sec:spectruminfty}
First, we establish in Theorem~\ref{thm:SpecGapInfty} the absence of eigenvalues of $A_X(\cdot)$ in the interval
$(-\infty,-c)$ for sufficiently large $c>0$. Later on in Theorem~\ref{thm:existencelambdapositiv}, we establish the
existence of an infinite sequence of positive eigenvalues of $A_X(\cdot)$ which accumulate at $+\infty$.

\subsection{The spectrum near negative infinity}
We require the following additional assumption for Theorem~\ref{thm:SpecGapInfty}.
\begin{assumption}[$\omega^2$ is no Dirichlet eigenvalue]\label{ass:NoDirichlet2}
There exists no non-trivial solution $u\in H(\curl,\tr_{\nv\times}^0;\Omega)$ to
\quad $\curl\mu^{-1}\curl u-\omega^2\epsilon u=0$ \quad in $\Omega$.
\end{assumption}

\begin{lemma}[Nitsche penalty technique]\label{lem:NitschePenalty}
Let Assumptions~\ref{ass:eps}, \ref{ass:mu}, \ref{ass:Domain} hold true.
Let $f\in \boldL^2(\Omega)$ and $u\in H(\curl,\tr_{\nv\times}^0;\Omega)$ be the solution to
\, $\curl\mu^{-1}\curl u+\epsilon u=f$ \quad in $\Omega$. For $\lambda>0$ let $u_\lambda\in
H(\curl,\tr_{\nv\times};\Omega)$ be the solution to
\begin{align*}
\spl \mu^{-1} \curl u_\lambda,\curl u' \spr_{\boldL^2(\Omega)} + \spl \epsilon u_\lambda, u' \spr_{\boldL^2(\Omega)}
+ \lambda \spl \tr_{\nv\times} u_\lambda, \tr_{\nv\times} u' \spr_{\boldL_t^2(\partial\Omega)}
= \spl f, u' \spr_{\boldL^2(\Omega)}
\end{align*}
for all $u'\in H(\curl,\tr_{\nv\times};\Omega)$. Then there exist $C,\lambda_0>0$ so that
\begin{align*}
\|u-u_\lambda\|_{H(\curl,\tr_{\nv\times};\Omega)} \leq C/\lambda
\end{align*}
for all $\lambda>\lambda_0$.
\end{lemma}
\begin{proof}
We are not aware of a direct appropriate reference for this lemma. Although we believe that the technique applied in
this proof is common knowledge. We introduce mixed equations for $u$ (and $u_\lambda$) as e.g.\ in \cite{Stenberg:95}
as follows. Let $\hat f\in X$ be so that $\spl \hat f,u'\spr_X=\spl f, u' \spr_{\boldL^2(\Omega)}$ for all $u'\in X$.
Due to $u\in H(\curl,\tr_{\nv\times}^0;\Omega)$ and Assumption~\ref{ass:mu} it follows
$\phi:=\nv\times\tr_{\nu\times} \mu^{-1}\curl u\in \boldL^2_t(\partial\Omega)$. It holds
$\phi_\lambda:=\nu\times\tr_{\nu\times} \mu^{-1}\curl u_\lambda=\lambda\tr_{\nu\times} u_\lambda
\in \boldL^2_t(\partial\Omega)$ too.
Integration by parts yields that $(u,\phi),(u_\lambda,\phi_\lambda) \in X\times \boldL^2_t(\partial\Omega)$ solve
\begin{align}\label{eq:Nitsche1}
\bpm A_c+A_\epsilon & B_{\tr}^* \\ B_{\tr} & 0 \epm
\bpm u\\\phi \epm = \bpm \hat f\\0 \epm
\end{align}
and
\begin{align}\label{eq:Nitsche2}
\bpm A_c+A_\epsilon & B_{\tr}^* \\ B_{\tr} & -\lambda^{-1} I_{\boldL^2_t(\partial\Omega)} \epm
\bpm u_\lambda\\\phi_\lambda \epm = \bpm \hat f\\0 \epm
\end{align}
respective. Both \eqref{eq:Nitsche1} and \eqref{eq:Nitsche2} are stable saddle point problems
\cite[Theorem~4.3.1]{BoffiBrezziFortin:13}. Since \eqref{eq:Nitsche2} is a perturbation of \eqref{eq:Nitsche1} by
magnitude $\lambda^{-1}$, the claim follows.
\end{proof}

\begin{theorem}\label{thm:SpecGapInfty}
Let Assumptions~\ref{ass:eps}, \ref{ass:mu}, \ref{ass:Domain}, \ref{ass:UCP} and~\ref{ass:NoDirichlet2} hold
true. Thence there exists $c>0$ so that $A_X(\lambda)$ is bijective for all $\lambda\in (-\infty,-c)$.
\end{theorem}
\begin{proof}
Assume the contrary. Thus there exists a sequence $(\lambda_n<0)_{n\in\setN}$ with $\lim_{n\in\setN} \lambda_n=-\infty$,
so that $A_X(\lambda_n)$ is not bijective. Due Theorem~\ref{thm:SpecBasic} $(\lambda_n)_{n\in\setN}$ are eigenvalues
of $A_X(\cdot)$. Hence let $(u_n\in X)_{n\in\setN}$ be a corresponding sequence of normalized eigenfunctions:
$A_X(\lambda_n)u_n=0$ and $\|u_n\|_X=1$ for each $n\in\setN$. It follows
\begin{align}\label{eq:uequals}
u_n=(\omega^2+1)(A_c+A_\epsilon+|\lambda_n| A_{\tr})^{-1} A_\epsilon u_n.
\end{align}
As already discussed at the end of Section~\ref{sec:evp}, it holds $u_n\in V\oplus W_1$ for each $n\in\setN$.
Denote $E \in L\big(X,\boldL^2(\Omega)\big)$ the embedding operator and $M_\epsilon\in L\big(\boldL^2(\Omega)\big)$ the multiplication
operator with symbol $\epsilon$. Thus $A_\epsilon = E^*M_\epsilon E$.
Due to Lemma~\ref{lem:compactEmbedding} there exist $f\in \boldL^2(\Omega)$ and a subsequence $(n(m))_{m\in\setN}$ so
that $\lim_{m\in\setN} Eu_{n(m)}=f$. Let $u\in H(\curl,\tr_{\nv\times}^0;\Omega)$ be the solution to
$\curl\mu^{-1}\curl u+\epsilon u=\epsilon f$ in $\Omega$. It follows from Lemma~\ref{lem:NitschePenalty} and
\eqref{eq:uequals} that $\lim_{m\in\setN} u_{n(m)}=(\omega^2+1)u$ in $X$. Since
\begin{align*}
\curl\mu^{-1}\curl u_{n(m)}-\omega^2\epsilon u_{n(m)}=0 \quad\text{ in }\Omega
\end{align*}
for each $m\in\setN$, it follows that
\begin{align*}
\curl\mu^{-1}\curl u-\omega^2\epsilon u=0 \quad\text{ in }\Omega
\end{align*}
as well. Due to Assumption~\ref{ass:NoDirichlet2} it holds $u=0$, which is a contradiction to $\|u_{n(m)}\|_X=1$
for each $m\in\setN$. The claim is proven.
\end{proof}

\subsection{The spectrum near positive infinity}
$P_{W_1} A_{\tr} |_{W_1}\in L(W_1)$ is strictly positive definite due to Lemma \ref{lem:PWAtrisPD}. Hence there exists
$c_\infty>0$ so that
\begin{align}
P_{W_1} (\omega^2 A_\epsilon + \lambda A_{\tr})|_{W_1}=\lambda P_{W_1} (\omega^2\lambda^{-1} A_\epsilon + A_{\tr})|_{W_1}
\end{align}
is coercive and thus bijective for each $\lambda\in\setC$ with $|\lambda|>c_\infty$.
(Since $A_\epsilon$ is positive semi definite, it follows even that $P_{W_1} (\omega^2 A_\epsilon + \lambda A_{\tr})|_{W_1}$
is coercive for each $\lambda\in\setC\setminus\setR^-_0$. However, we will not use this fact.)
Hence for $|\lambda|>c_\infty$ we build and study the Schur complement of
$(P_V+P_{W_1}) A_X(\lambda) |_{V\oplus W_1}$ with respect to $P_{W_1}u=w_1$:
\begin{subequations}
\begin{align}
A_V(\lambda)&:=P_V (A_c -\omega^2 A_\epsilon) |_V-\lambda K_V(\lambda) \in L(V),\\
K_V(\lambda)&:=P_V(A_{\tr}-A_{\tr}S_V(\lambda)P_{W_1}A_{\tr})|_V  \in L(V),\\
S_V(\lambda)&:=\Big(P_{W_1}(\omega^2 \lambda^{-1}A_\epsilon + A_{\tr} )|_{W_1}\Big)^{-1}  \in L(W_1).
\end{align}
\end{subequations}
It is straight forward to see, that for $\lambda$ satisfying $|\lambda|>c_\infty$, $\lambda$ is an eigenvalue to
$A_X(\cdot)$ if and only if $\lambda$ is an eigenvalue to $A_V(\cdot)$. Hence to study the eigenvalues of $A_X(\cdot)$
in a neighborhood of infinity, it completely suffices to study the eigenvalues of $A_V(\cdot)$ in a neighborhood of
infinity. It will be more convenient to work with $\lambda^{-1}$ instead of $\lambda$. Hence let
\begin{subequations}
\begin{align}
\tilde A_V(\tilde \lambda)&:=\tilde \lambda A_V(\tilde\lambda^{-1})
=\tilde \lambda P_V (A_c -\omega^2 A_\epsilon) |_V -\tilde K_V(\tilde\lambda) \in L(V),\\
\tilde K_V(\tilde\lambda)&:=K_V(\tilde\lambda^{-1})
=P_V(A_{\tr}-A_{\tr}\tilde S_V(\tilde \lambda)P_{W_1}A_{\tr})|_V \in L(V),\\
\tilde S_V(\tilde\lambda)&:=S_V(\tilde\lambda^{-1})=
\Big(P_{W_1}(\omega^2 \tilde\lambda A_\epsilon + A_{\tr} )|_{W_1}\Big)^{-1} \in L(W_1),
\end{align}
\end{subequations}
for $\tilde\lambda\in\setC$ with $|\tilde\lambda|<c_\infty^{-1}$. Again, it is straight forward to see that
$\tilde\lambda$ with $0<|\tilde\lambda|<c_\infty^{-1}$ is an eigenvalue to $\tilde A_V(\cdot)$ if and only if
$\tilde\lambda^{-1}$ with $|\tilde\lambda^{-1}|>c_\infty$ is an eigenvalue to $A_V(\cdot)$. Thus we study the
eigenvalues of $\tilde A_V(\cdot)$ in the ball
\begin{align}
B_{c_\infty^{-1}}:=\{z\in\setC\colon |z|<c_\infty^{-1}\}.
\end{align}
To this end we introduce
\begin{align}\label{eq:tAVtaulambda}
\tilde A_V(\tilde\tau,\tilde\lambda):=\tilde \tau P_V (A_c -\omega^2 A_\epsilon) |_V -\tilde K_V(\tilde \lambda).
\end{align}
We note that $\tilde\lambda\in B_{c_\infty^{-1}}$ is an eigenvalue of $\tilde A_V(\cdot)$, if and only if $\tilde\tau$
is an eigenvalue of $\tilde A_V(\cdot,\tilde\lambda)$ and $\tilde\tau=\tilde\lambda\in B_{c_\infty^{-1}}$.

We would like to proceed as in Section~\ref{sec:spectrumzero}. Operator $\tilde K_V(\tilde\lambda)$ is compact due
Lemma~\ref{lem:Vtraceregularity}. However different to Section~\ref{sec:spectrumzero}, $P_V (A_c -\omega^2 A_\epsilon) |_V$
is (for arbitrary $\omega>0$) not definite! Moreover, $\tilde K_V(\tilde\lambda)$ is not injective! Indeed
$\{\curl f\colon f\in(C^\infty_0(\hat\Omega \setminus \check\Omega))^3\}\subset \ker \tilde K_V(\tilde\lambda)$.
Therefore, we introduce the abstract Lemma~\ref{lem:abstract}. Subsequently, we prove that the conditions of
Lemma~\ref{lem:abstract} are satisfied and the lemma can be employed for our particular application.
We derive the results aimed at in Lemma~\ref{lem:existencetaunV} and consequently continue the analysis in the same
manner as in Section~\ref{sec:spectrumzero}.

\begin{lemma}\label{lem:abstract}
Let $Y$ be a separable Hilbert space. Let $G\in L(Y)$ be compact, selfadjoint and $I+G$ be bijective.
Let $K\in L(Y)$ be compact, selfadjoint, positive semi definite and so that
$\ker K=\ker (K^{1/2} (I+G) K^{1/2})$ and $\dim (\ker K)^\bot=\infty$.
Let $P_{(\ker K)^\bot}$ be the orthogonal projection onto $(\ker K)^\bot$ and $P_{(\ker K)^\bot}(I+G)|_{(\ker K)^\bot}$
be bijective.\\
Then the spectra of $(I+G)K$ and $K^{1/2} (I+G) K^{1/2}$ coincide and consist of the essential spectrum $\{0\}$ and
an infinite sequence $(\tau_n\in\setR)_{n\in\setN}$ of non-zero eigenvalues.
Apart from a finite set all $(\tau_n)_{n\in\setN}$ are positive and it holds $\lim_{n\in\setN} \tau_n=0$.
\end{lemma}
\begin{proof}
\textit{1.\ Step:}\quad If $(\tau,y)\in\setC\setminus\{0\}\times Y\setminus\{0\}$ solves
\begin{align*}
\Big(\tau I-(I+G)K\Big)y=0,
\end{align*}
then $K^{1/2}y\neq0$ and
\begin{align*}
0=K^{1/2}\Big(\tau I-(I+G)K\Big)y=\Big(\tau I-K^{1/2}(I+G)K^{1/2}\Big)K^{1/2}y.
\end{align*}
Vice-versa, if $(\tau,y')\in\setC\setminus\{0\}\times Y\setminus\{0\}$ solves
\begin{align*}
\Big(\tau I-K^{1/2}(I+G)K^{1/2}\Big)y',
\end{align*}
then $(I+G)K^{1/2}y'\neq0$ and
\begin{align*}
0=\Big(I+G)K^{1/2}(\tau I-K^{1/2}(I+G)K^{1/2}\Big)y'=\Big(\tau I-(I+G)K\Big)(I+G)K^{1/2}y'.
\end{align*}
By assumption, $(I+G)Ky=0$ if and only if $K^{1/2}(I+G)K^{1/2}y=0$. Thus the spectra of $(I+G)K$ and
$K^{1/2}(I+G)K^{1/2}$ coincide.\\
\textit{2.\ Step:}\quad Since $K^{1/2}(I+G)K^{1/2}$ is compact and selfadjoint and $Y$ is separable with
$\dim Y\geq \dim (\ker K)^\bot=\infty$ the Spectral Theorem for compact, selfadjoint operators yields:
The spectrum of $K^{1/2}(I+G)K^{1/2}$ consists of the essential spectrum $\{0\}$ and an infinite sequence of eigenvalues
$(\tau_n\in\setR)_{n\in\setN}$ (with multiplicity taken into account),
$\lim_{n\in\setN} \tau_n=0$ and there exists an orthonormal basis $(y_n)_{n\in\setN}$ of corresponding eigenelements.
Due to $\dim (\ker K)^\bot=\infty$ there exists an infinite index set $\setM\subset\setN$ so that $\tau_m\neq0$ for each
$m\in\setM$.\\
\textit{3.\ Step:}\quad It remains to prove that all $(\tau_m)_{m\in\setM}$ apart from a finite set are positive.
To this end we apply a technique which is inspired by \cite[\S 3]{Markus:88}. Let
\begin{align*}
\tilde Y:=\ol{\spn \{y_m\colon m\in\setM\}}^\mathrm{cl}=(\ker K^{1/2}(I+G)K^{1/2})^{\bot}=(\ker K)^\bot
\end{align*}
and denote $P_{\tilde Y}$ the orthogonal projection onto $\tilde Y$.
We note that for each $y\in Y$, $y^0\in\ker K$ it holds
\begin{align*}
\spl K^{1/2}y,y^0\spr_Y = \spl y,K^{1/2}y^0\spr_Y = 0.
\end{align*}
Thus $\ran K^{1/2} \subset (\ker K)^{\bot}=\tilde Y$ and so $(\tau I+K)^{1/2}\tilde Y\subset \tilde Y$.
Let $G=G_+-G_-$ so that $G_+$ and $G_-$ are compact, selfadjoint and positive semi definite, i.e.\ a decomposition of
$G$ in the positive and the negative part. For $\tau>0$ we compute
\begin{align*}
\big(\tau I&+K^{1/2}(I+G)K^{1/2}\big)|_{\tilde Y}\\
&=\big(\tau I + K + K^{1/2}GK^{1/2}\big)|_{\tilde Y}\\
&=(\tau I+K)^{1/2} \Big( I -(P_{\tilde Y}(\tau I+K)|_{\tilde Y})^{-1/2}\\
&\phantom{=}K^{1/2}(G_+^{1/2}G_+^{1/2}-G_-^{1/2}G_-^{1/2})
K^{1/2} (P_{\tilde Y}(\tau I+K)|_{\tilde Y})^{-1/2} \Big) (\tau I+K)^{1/2}|_{\tilde Y}.
\end{align*}
By means of the Spectral Theorem for compact, selfadjoint operators we deduce that
$(P_{\tilde Y}(\tau I+K)|_{\tilde Y})^{-1/2}K^{1/2}$ converges point-wise to $P_{\tilde Y}$ for $\tau\to0+$. Since
$G_\pm^{1/2}$ is compact it follows that $(P_{\tilde Y}(\tau I+K)|_{\tilde Y})^{-1/2}K^{1/2}G_\pm^{1/2}$ converges to
$P_{\tilde Y}G_\pm^{1/2}$ in $L(Y)$ for $\tau\to0+$. Hence
\begin{align*}
\Big( (P_{\tilde Y}(\tau I+K)|_{\tilde Y})^{-1/2}K^{1/2}G_\pm^{1/2} \Big)^*
=G_\pm^{1/2}K^{1/2}(P_{\tilde Y}(\tau I+K)|_{\tilde Y})^{-1/2}P_{\tilde Y}
\end{align*}
converges to  $(P_{\tilde Y}G_\pm^{1/2})^*= G_\pm^{1/2}P_{\tilde Y}$ in $L(Y)$. Thus
\begin{align}\label{eq:proofabstract}
P_{\tilde Y} \Big(
I -(P_{\tilde Y}(\tau I+K)|_{\tilde Y})^{-1/2}K^{1/2} G K^{1/2} (P_{\tilde Y}(\tau I+K)|_{\tilde Y})^{-1/2}
\Big) |_{\tilde Y}
\end{align}
converges in norm to $P_{\tilde Y}(I-G)|_{\tilde Y}$. 
Hence there exists $c>0$ so that \eqref{eq:proofabstract} is bijective for all $\tau\in(0,c)$. Since for each
$\tau\in(0,c)$, $(\tau I+K^{1/2}(I+G)K^{1/2})|_{\tilde Y} \in L(\tilde Y)$ is a composition of three bijective operators
in $L(\tilde Y)$, it is bijective. Due to $\lim_{m\in\setM}\tau_m=0$ there can only exist a finite number of $m\in\setM$
with $\tau_m<0$.
\end{proof}

\begin{lemma}\label{lem:tKVpsd}
Let Assumptions~\ref{ass:eps}, \ref{ass:mu}, \ref{ass:Domain} hold true.
Thence $\tilde K_V(\tilde\lambda)$ is compact, selfadjoint and positive semi definite for each
$\tilde\lambda\in[0,c_\infty^{-1})$.
It holds further $\ker \tilde K_V(\tilde\lambda)=\ker B_{\tr}$ for each $\tilde\lambda\in(0,c_\infty^{-1})$.
\end{lemma}
\begin{proof}
Let $\tilde\lambda\in[0,c_\infty^{-1})$. $\tilde K_V(\tilde\lambda)$ is compact due Lemma~\ref{lem:Vtraceregularity}.
It follows from the definition of $\tilde K_V(\tilde\lambda)$, that $\tilde K_V(\tilde\lambda)$ is selfadjoint.
Let $v\in V$ and $w_1:=\tilde S_V(\tilde\lambda) P_{W_1} B_{\tr}^* B_{\tr} v$. We compute
\begin{align*}
\spl B_{\tr}w_1,B_{\tr}w_1 \spr_{\boldL_t^2(\partial\Omega)}
&\leq \spl B_{\tr}w_1,B_{\tr}w_1 \spr_{\boldL_t^2(\partial\Omega)}
+\omega^2\tilde\lambda \spl \epsilon w_1,w_1 \spr_{\boldL^2(\Omega)}\\
&=\spl (A_{\tr}+\omega^2\tilde\lambda A_\epsilon)w_1,w_1\spr_X \\
&=\spl (A_{\tr}+\omega^2\tilde\lambda A_\epsilon)\tilde S_V(\tilde\lambda) P_{W_1} B_{\tr}^* B_{\tr}v,w_1\spr_X \\
&=\spl B_{\tr}v,B_{\tr}w_1 \spr_{\boldL_t^2(\partial\Omega)}\\
&\leq \|B_{\tr}v\|_{\boldL_t^2(\partial\Omega)}
\|B_{\tr}w_1\|_{\boldL_t^2(\partial\Omega)}
\end{align*}
and hence $\|B_{\tr}\tilde S_V(\tilde\lambda) P_{W_1} B_{\tr}^* B_{\tr} v\|_{\boldL_t^2(\partial\Omega)}
=\|B_{\tr}w_1\|_{\boldL_t^2(\partial\Omega)}
\leq \|B_{\tr}v\|_{\boldL_t^2(\partial\Omega)}$. Thus
\begin{align*}
\spl B_{\tr} \tilde S_V(\tilde\lambda) P_{W_1} B_{\tr}^* B_{\tr} v,
B_{\tr}v \spr_{\boldL^2_t(\partial\Omega)}
&\leq \|B_{\tr} \tilde S_V(\tilde\lambda) P_{W_1} B_{\tr}^* B_{\tr} v\|_{\boldL^2_t(\partial\Omega)}
\|B_{\tr}v \|_{\boldL^2_t(\partial\Omega)}\\
&\leq \|B_{\tr}v\|_{\boldL_t^2(\partial\Omega)}
\|B_{\tr}v \|_{\boldL^2_t(\partial\Omega)}.
\end{align*}
Hence
\begin{align}\label{eq:tKpsd}
\begin{aligned}
\spl \tilde K_V(\tilde\lambda)v,v\spr_X
&= \spl (A_{\tr}-A_{\tr} \tilde S_V(\tilde\lambda) P_{W_1} A_{\tr}v,v\spr_X\\
&= \spl B_{\tr}v,B_{\tr}v \spr_{\boldL^2_t(\partial\Omega)}
-\spl B_{\tr} \tilde S_V(\tilde\lambda) P_{W_1} B^*_{\tr} B_{\tr}v, B_{\tr}v \spr_{\boldL^2_t(\partial\Omega)}\\
&\geq0.
\end{aligned}
\end{align}
Let $\tilde\lambda\in(0,c_\infty^{-1})$.
Let $B_{\tr}v\neq0$. If $P_{W_1}B_{\tr}^*B_{\tr}v=0$ it follows $\tilde S_V(\tilde\lambda)P_{W_1}B_{\tr}^*B_{\tr}v=0$
and \eqref{eq:tKpsd} is strict. So let $P_{W_1}B_{\tr}^*B_{\tr}v\neq0$. It follows $w_1\neq0$ and hence
$\spl \epsilon w_1,w_1 \spr_{\boldL^2(\Omega)}>0$. Since $w_1\in W_1$, it also holds
$\|B_{\tr}w_1\|_{\boldL^2_t(\partial\Omega)}\neq0$. So in this case
$\|B_{\tr}\tilde S_V(\tilde\lambda) P_{W_1} B_{\tr}^* B_{\tr} v\|_{\boldL_t^2(\partial\Omega)}$
$<$ $\|B_{\tr}v\|_{\boldL_t^2(\partial\Omega)}$ and \eqref{eq:tKpsd} is strict too. Thus $\tilde K_V(\tilde\lambda)v\neq0$.
On the other hand: If $B_{\tr}v=0$, then also $\tilde K_V(\tilde\lambda)v=0$ due to the definition of
$\tilde K_V(\tilde\lambda)$. Thus $\ker \tilde K_V(\tilde\lambda)=\ker B_{\tr}$ for each
$\tilde\lambda\in(0,c_\infty^{-1})$.
\end{proof}

\begin{lemma}\label{lem:tKVzero}
Let Assumptions~\ref{ass:eps}, \ref{ass:mu}, \ref{ass:Domain} hold true. Thence
\begin{align*}
\tilde K_V(0)=B_{\tr}^* P_{\nabla_\partial} B_{\tr}.
\end{align*}
\end{lemma}
\begin{proof}
Let $P\in L\big(\boldL^2_t(\partial\Omega)\big)$ be the $\boldL^2_t(\partial\Omega)$-orthogonal
projection onto the closure of $\ran B_{\tr}|_{W_1}$. It follows from the definition of $\tilde K_V(0)$ that
$\tilde K_V(0)=B_{\tr}^* (I-P) B_{\tr}$. The claim is proven, if we show that $\ran B_{\tr}|_{W_1}=
\curl_\partial H^1(\partial\Omega)$. It follows from the definition of $W_1$ that $\ran B_{\tr}|_{W_1}\subset
\curl_\partial H^1(\partial\Omega)$. Let $\phi\in\curl_\partial H^1(\partial\Omega)$ and $\psi\in H^1(\partial\Omega)$
so that $\phi=\curl_\partial\psi=\nu\times\nabla_\partial\psi$. Let $\tilde w\in H^1(\Omega)$ solve
$\diveps\nabla\tilde w=0$ in $\Omega$ and $\tr\tilde w=\psi$ at $\partial\Omega$. With \eqref{eq:W2char} it follow
$\nabla\tilde w-P_{W_2}\nabla\tilde w=:w\in W_1$ and $B_{\tr}w=\phi$. Thus $\ran B_{\tr}|_{W_1}=
\curl_\partial H^1(\partial\Omega)$ and
\begin{align*}
\tilde K_V(0)=B_{\tr}^* (I-P) B_{\tr}=B_{\tr}^* (I-P_{\nabla_\partial^\bot}) B_{\tr}=B_{\tr}^* P_{\nabla_\partial} B_{\tr}.
\end{align*}
\end{proof}

\begin{lemma}\label{lem:dimVfkerPB}
Let Assumptions~\ref{ass:eps}, \ref{ass:mu}, \ref{ass:Domain} hold true.
Thence
\begin{align*}
\dim (\ker B_{\tr}|_V)^{\bot_V} = \dim (\ker P_{\nabla_\partial}B_{\tr}|_V)^{\bot_V}=\infty.
\end{align*}
\end{lemma}
\begin{proof}
Let $(f_n)_{n\in\setN}$ be an orthonormal basis of $\nabla_\partial H^1(\partial\Omega)\subset\boldL^2_t(\partial\Omega)$.
Let $u_n\in H(\curl;\Omega)$ be so that $\tr_{\nu\times}u_n=f_n$. Hence $u_n\in X$. It follows
\begin{align*}
P_{\nabla_\partial} B_{\tr} (P_Vu_n+\ker P_{\nabla_\partial}B_{\tr}|_V)=f_n.
\end{align*}
Thus if $\sum_{n=1}^N c_n(P_Vu_n+\ker P_{\nabla_\partial}B_{\tr}|_V)$ would be a non-trivial linear combination of zero in
$V/(\ker P_{\nabla_\partial}B_{\tr}|_V)$, then $\sum_{n=1}^N c_n f_n$ would be a non-trivial linear combination of zero
in $\nabla_\partial H^1(\partial\Omega)$. Hence $\dim V/(\ker P_{\nabla_\partial}B_{\tr}|_V)=+\infty$.
Since $\ker B_{\tr}\subset\ker P_{\nabla_\partial}B_{\tr}$ it follows
$\dim V/(\ker B_{\tr}|_V) \geq \dim V/(\ker P_{\nabla_\partial}B_{\tr}|_V)$ and thus the dimension of
$\dim V/(\ker B_{\tr}|_V$ is infinite too. The claim follows from $\dim V/Z=\dim Z^{\bot_V}$ for any closed subspace
$Z\subset V$.
\end{proof}

We require the following additional assumption for Lemma~\ref{lem:equalkernelsNZ}.
\begin{assumption}[$\omega^2$ is no ``Dirichlet'' eigenvalue]\label{ass:NoDirichlet}
Let
\begin{align*}
Z_1:=\{z\in V\colon B_{\tr}z=0\}=H(\curl,\diveps^0,\tr_{\nv\times}^0,\tr_{\nv\cdot\epsilon}^0;\Omega)
\end{align*}
and denote $P_{Z_1}$ the $X$-orthogonal projection onto $Z_1$. The operator
\begin{align*}
P_{Z_1} A_c |_{Z_1}  -\omega^2 P_{Z_1} A_\epsilon |_{Z_1} \in L(Z_1)
\end{align*}
is bijective.
\end{assumption}

\begin{lemma}\label{lem:equalkernelsNZ}
Let Assumptions~\ref{ass:eps}, \ref{ass:mu}, \ref{ass:Domain}, \ref{ass:NoNeumann} and \ref{ass:NoDirichlet} hold true.
Let $\tilde\lambda\in(0,c_\infty^{-1})$. Thence
\begin{align*}
\ker \Big(\tilde K_V(\tilde\lambda)^{1/2} \big(P_V (A_c -\omega^2 A_\epsilon) |_V\big)^{-1}
\tilde K_V(\tilde\lambda)^{1/2}\Big) = \ker \tilde K_V(\tilde\lambda).
\end{align*}
\end{lemma}
\begin{proof}
Let $v\in \ker \Big(\tilde K_V(\tilde\lambda)^{1/2} \big(P_V (A_c -\omega^2 A_\epsilon) |_V\big)^{-1}
\tilde K_V(\tilde\lambda)^{1/2}\Big)$ and
\begin{align*}
z:=\big(P_V (A_c -\omega^2 A_\epsilon) |_V\big)^{-1} \tilde K_V(\tilde\lambda)^{1/2}\Big)v. 
\end{align*}
It follows $B_{\tr}z=0$ due to $\ker K_V(\tilde\lambda)^{1/2} = \ker K_V(\tilde\lambda)$ and Lemma~\ref{lem:tKVpsd}.
Due to the definitions of $z$ and $Z_1$, $z\in Z_1$ solves
\begin{align*}
(P_{Z_1} A_c  -\omega^2 P_{Z_1} A_\epsilon)z=0.
\end{align*}
It follows from Assumption~\ref{ass:NoDirichlet} that $z=0$. Thus $v\in \ker K_V(\tilde\lambda)^{1/2}=
\ker K_V(\tilde\lambda)$.
\end{proof}

We require the following additional assumption for Lemma~\ref{lem:equalkernelsZ}.
\begin{assumption}[$\omega^2$ is no ``hybrid'' eigenvalue]\label{ass:NoHybrid}
Let
\begin{align*}
Z_2:=\{z\in V\colon P_{\nabla_\partial} B_{\tr}z=0\}
\end{align*}
and denote $P_{Z_2}$ the $X$-orthogonal projection onto $Z_2$. The operator
\begin{align*}
P_{Z_2} A_c |_{Z_2}  -\omega^2 P_{Z_2} A_\epsilon |_{Z_2} \in L(Z_2)
\end{align*}
is bijective.
\end{assumption}

\begin{lemma}\label{lem:equalkernelsZ}
Let Assumptions~\ref{ass:eps}, \ref{ass:mu}, \ref{ass:Domain}, \ref{ass:NoNeumann} and \ref{ass:NoHybrid} hold true.
Thence
\begin{align*}
\ker \Big(\tilde K_V(0)^{1/2} \big(P_V (A_c -\omega^2 A_\epsilon) |_V\big)^{-1} \tilde K_V(0)^{1/2}\Big)
= \ker \tilde K_V(0).
\end{align*}
\end{lemma}
\begin{proof}
Let $v\in \ker \Big(\tilde K_V(0)^{1/2} \big(P_V (A_c -\omega^2 A_\epsilon) |_V\big)^{-1} \tilde K_V(0)^{1/2}\Big)$ and
\begin{align*}
z:=\big(P_V (A_c -\omega^2 A_\epsilon) |_V\big)^{-1} \tilde K_V(0)^{1/2}\Big)v. 
\end{align*}
It follows $P_{\nabla_\partial} B_{\tr}z=0$ due to $\ker K_V(0)^{1/2} = \ker K_V(0)$ and Lemma~\ref{lem:tKVzero}.
Due to the definitions of $z$ and $Z_2$, $z\in Z_2$ solves
\begin{align*}
(P_{Z_2} A_c  -\omega^2 P_{Z_2} A_\epsilon)z=0.
\end{align*}
It follows from Assumption~\ref{ass:NoHybrid} that $z=0$. Thus $v\in \ker K_V(0)^{1/2}=\ker K_V(0)$.
\end{proof}

We require the following additional assumption for Lemma~\ref{lem:existencetaunV}.
\begin{assumption}[$\omega^2$ is no ``projected'' eigenvalue]\label{ass:NoReduced}
The operators
\begin{align*}
P_{Z_1} \big(P_V (A_c -\omega^2 A_\epsilon) |_V\big)^{-1} |_{Z_1} \in L(Z_1)
\end{align*}
and
\begin{align*}
P_{Z_2} \big(P_V (A_c -\omega^2 A_\epsilon) |_V\big)^{-1} |_{Z_2} \in L(Z_2)
\end{align*}
are bijective.
\end{assumption}

We note that
\begin{align}
P_V A_c |_V= P_V (I - A_\epsilon - A_{\tr} )|_V
\end{align}
and consequently
\begin{align}\label{eq:DefG}
\begin{aligned}
P_V (A_c -\omega^2 A_\epsilon) |_V\big)^{-1} &=
I|_V - P_V (A_c -\omega^2 A_\epsilon) |_V\big)^{-1} P_V ( (\omega^2+1)A_\epsilon + A_{\tr} )|_V\\
&=:I_V+G.
\end{aligned}
\end{align}

\begin{lemma}\label{lem:existencetaunV}
Let Assumptions~\ref{ass:eps}, \ref{ass:mu}, \ref{ass:Domain}, \ref{ass:UCP} and
\ref{ass:NoNeumann}, \ref{ass:NoDirichlet}, \ref{ass:NoHybrid}, \ref{ass:NoReduced} hold true.
Let $\tilde\lambda\in[0,c_\infty^{-1})$. The spectrum of $\tilde A_V(\cdot,\tilde\lambda)$ consists of
$\sigma_\mathrm{ess}\big(\tilde A_V(\cdot,\tilde\lambda)\big)=\{0\}$ and an infinite sequence of non-zero eigenvalues
$(\tilde\tau_n(\tilde\lambda))_{n\in\setN}$ with $\lim_{n\in\setN}\tilde\tau_n(\tilde\lambda)=0$.
Apart from a finite number, all non-zero eigenvalues $(\tilde\tau_n(\tilde\lambda))_{n\in\setN}$ are positive.
\end{lemma}
\begin{proof}
We note that $\tilde A_V(\cdot,\tilde\lambda)$ and
$P_V (A_c -\omega^2 A_\epsilon) |_V\big)^{-1} \tilde A_V(\cdot,\tilde\lambda)$ have the very same spectral properties.
We aim to apply Lemma~\ref{lem:abstract} to
\begin{align*}
P_V (A_c -\omega^2 A_\epsilon) |_V\big)^{-1} \tilde A_V(\cdot,\tilde\lambda)
= \tilde\tau I- (I+G)\tilde K_V(\tilde\lambda)
\end{align*}
with $G$ defined as in \eqref{eq:DefG}. $G$ is compact due to Lemma~\ref{lem:compactEmbedding} and
Lemma~\ref{lem:Vtraceregularity}. Since $P_V (A_c -\omega^2 A_\epsilon) |_V\big)^{-1}$ and the identity are selfadjoint,
the selfadjointness of $G$ follows from \eqref{eq:DefG}. $I+G$ is bijective due to its definition and
Assumption~\ref{ass:NoNeumann}. $\tilde K_V(\tilde\lambda)$ is compact, selfadjoint and positive semi definite due to
Lemma~\ref{lem:tKVpsd}. It holds
$\ker \Big(\tilde K_V(\tilde\lambda)^{1/2} (I+G) \tilde K_V(\tilde\lambda)^{1/2}\Big) = \ker \tilde K_V(\tilde\lambda)$
due to Lemma~\ref{lem:equalkernelsNZ} and Lemma~\ref{lem:equalkernelsZ}. It holds
$\dim (\ker \tilde K_V(\tilde\lambda))^\bot=\infty$ due to Lemma~\ref{lem:dimVfkerPB}. 
$P_{(\ker \tilde K_V(\tilde\lambda))^\bot}(I+G)|_{(\ker \tilde K_V(\tilde\lambda))^\bot}$ is bijective due to
Assumption~\ref{ass:NoReduced}. Hence the conditions of Lemma~\ref{lem:abstract} are satisfied and the claim follows.
\end{proof}

\begin{lemma}\label{lem:continuitytaunV}
Let Assumptions~\ref{ass:eps}, \ref{ass:mu}, \ref{ass:Domain} and
\ref{ass:NoNeumann}, \ref{ass:NoDirichlet}, \ref{ass:NoHybrid}, \ref{ass:NoReduced} hold true.
For $\tilde\lambda\in[0,c_\infty^{-1})$ let $(\tilde\tau_n^+(\tilde\lambda))_{n\in\setN}$ be a non-increasing ordering
with multiplicity taken into account of the positive eigenvalues of $\tilde A_V(\cdot,\tilde\lambda)$. Thence for each
$n\in\setN$ the function $\tilde\tau^+_n\colon [0,c_\infty^{-1}) \to \setR^+$ is continuous.
\end{lemma}
\begin{proof}
We note that for each $n\in\setN$ it holds $\inf_{\tilde\lambda\in [0,c_\infty^{-1})}\tilde\tau^+_n(\tilde\lambda)>0$: Indeed
the existence of $\tilde\lambda_0\in[0,c_\infty^{-1})$, $n\in\setN$ so that $\lim_{\tilde\lambda\to\tilde\lambda_0+}\tilde\tau^+_n
(\tilde\lambda)=0$ would imply that for $\tilde\lambda=\tilde\lambda_0$ there would exist only a finite number of
positive eigenvalues, which is a contradiction to Lemma~\ref{lem:existencetaunV}.
The continuity of $\tilde\tau_n^+$ follows with \cite[Proposition~5.4]{SanchezSanchez:89}.
We note that a delicate part of \cite[Proposition~5.4]{SanchezSanchez:89} is the existence of eigenvalues. However, the
existence of eigenvalues is already established by Lemma~\ref{lem:existencetaunV}. We only require the continuity result
of \cite[Proposition~5.4]{SanchezSanchez:89}.
\end{proof}

\begin{theorem}\label{thm:existencelambdapositiv}
Let Assumptions~\ref{ass:eps}, \ref{ass:mu}, \ref{ass:Domain}, \ref{ass:UCP} and
\ref{ass:NoNeumann}, \ref{ass:NoDirichlet}, \ref{ass:NoHybrid}, \ref{ass:NoReduced} hold true.
Thence there exists an infinite sequence $(\lambda_n)_{n\in\setN}$ of positive eigenvalues to $A_X(\cdot)$ which
accumulate at $+\infty$.
\end{theorem}
\begin{proof}
Proceed as in the proof of Theorem~\ref{thm:existencelambdanegative}.
\end{proof}

\section{Conclusion}\label{sec:conclusion}
We conclude with a summary of Theorems~\ref{thm:SpecBasic}, \ref{thm:SpecGapZero}, \ref{thm:existencelambdanegative},
\ref{thm:SpecGapInfty} and \ref{thm:existencelambdapositiv} and some remarks on assumptions and the relation to the
modified electromagnetic Stekloff eigenvalue considered in \cite{Halla:19StekloffAppr}, \cite{CamanoLacknerMonk:17}.

\subsection{Main result}
We formulate the individual results of the previous sections in the following proposition.
\begin{proposition}\label{prop:mainprop}
Let Assumptions~\ref{ass:eps}, \ref{ass:mu}, \ref{ass:Domain}, \ref{ass:UCP} and \ref{ass:NoNeumann}, \ref{ass:NoDirichlet2},
\ref{ass:NoDirichlet}, \ref{ass:NoHybrid}, \ref{ass:NoReduced} be satisfied. Then it hold
\begin{align}
\sigma\big(A_X(\cdot)\big) &= \sigma_\mathrm{ess}\big(A_X(\cdot)\big) \dot\cup \bigcup_{n\in\setN} \{\lambda_n^{-0}\}
\dot\cup \bigcup_{n\in\setN} \{\lambda_n^{+\infty}\}
\end{align}
and $\sigma_\mathrm{ess}\big(A_X(\cdot)\big) = \{0\}$. The sequence $(\lambda_n^{-0})_{n\in\setN}$ consists of
pair-wise distinct negative eigenvalues with finite algebraic multiplicity so that $\lim_{n\in\setN} \lambda_n^{-0}=0$.
The sequence $(\lambda_n^{+\infty})_{n\in\setN}$ consists of pair-wise distinct positive eigenvalues with finite algebraic
multiplicity so that $\lim_{n\in\setN} \lambda_n^{+\infty}=+\infty$.
\end{proposition}
\begin{proof}
Follows from Theorems~\ref{thm:SpecBasic}, \ref{thm:SpecGapZero}, \ref{thm:existencelambdanegative},
\ref{thm:SpecGapInfty} and \ref{thm:existencelambdapositiv}.
\end{proof}

\subsection{Remarks to the assumptions}
The condition in Assumptions \ref{ass:eps} and \ref{ass:mu} that $\mu$ and $\epsilon$ equal the identity matrix
in a neighborhood of the boundary is used to obtain extra regularity of traces. If this extra regularity can be derived
by other means, then the mentioned assumption becomes obsolete.

Each of the Assumptions \ref{ass:NoNeumann}, \ref{ass:NoDirichlet2}, \ref{ass:NoDirichlet}, \ref{ass:NoHybrid},
\ref{ass:NoReduced} can be formulated in the following manner: $Y$ is a Hilbert space, $A\in L(Y)$ is weakly coercive,
$K(\cdot)\colon \Lambda\subset\setC\to K(Y)$ is holomorphic and it is imposed that $A-K(\omega^2)$ is bijective.
Consequently for fixed domain $\Omega$ and fixed material parameters $\mu^{-1}, \epsilon$ there exists only a countable
set of frequencies $\omega$ for which the Assumptions \ref{ass:NoNeumann}, \ref{ass:NoDirichlet2}, \ref{ass:NoDirichlet},
\ref{ass:NoHybrid}, \ref{ass:NoReduced} are not satisfied (see e.g.\ \cite[Proposition~A.8.4]{KozlovMazya:99}).

\subsection{Modified electromagnetic Stekloff eigenvalues}
The modified electromagnetic Stekloff eigenvalue problem considered in \cite{Halla:19StekloffAppr} is to find
$(\lambda,u)\in\setC\times H(\curl;\Omega)\setminus\{0\}$ so that
\begin{align}
\spl \mu^{-1}\curl u,\curl u'\spr_{\boldL^2(\Omega)}-\omega^2 \spl \epsilon u, u'\spr_{\boldL^2(\Omega)}
-\lambda \spl Su,Su'\spr_{\boldL^2_t(\partial\Omega)} =0
\end{align}
for all $u'\in H(\curl;\Omega)$ (with $S$ defined as in \eqref{eq:DefS}).
It can easily be seen that the eigenvalue problem decouples with respect to the decomposition
$H(\curl;\Omega)=H(\curl,\diveps^0,\tr_{\nv\cdot\epsilon}^0;\Omega)\oplus\nabla H^1(\Omega)$.
Thus the eigenvalue problem can be reformulated to find
$(\lambda,u)\in\setC\times H(\curl,\diveps^0,\tr_{\nv\cdot\epsilon}^0;\Omega)\setminus\{0\}$ so that
\begin{align}
\begin{aligned}
0&=
\spl \mu^{-1}\curl u,\curl u'\spr_{\boldL^2(\Omega)}-\omega^2 \spl \epsilon u, u'\spr_{\boldL^2(\Omega)}
-\lambda \spl Su,Su'\spr_{\boldL^2_t(\partial\Omega)}\\
&=\spl \mu^{-1}\curl u,\curl u'\spr_{\boldL^2(\Omega)}-\omega^2 \spl \epsilon u, u'\spr_{\boldL^2(\Omega)}
-\lambda \spl P_{\nabla_\partial} \tr_{\nv\times}u,\tr_{\nv\times}u'\spr_{\boldL^2_t(\partial\Omega)}\\
&= \spl \lambda \tilde A_V(\lambda^{-1},0) u,u'\spr_X
\end{aligned}
\end{align}
for all $u' \in H(\curl,\diveps^0,\tr_{\nv\cdot\epsilon}^0;\Omega)$. Thence if the respective assumptions are satisfied,
Lemma~\ref{lem:existencetaunV} yields that the spectrum consists of an infinite sequence of eigenvalues
$(\lambda_n)_{n\in\setN}$ which accumulate only at $+\infty$.
A similar existence result has been reported in \cite[Theorem~3.6]{CamanoLacknerMonk:17}. Though it seems to us that the
proof of \cite[Theorem~3.6]{CamanoLacknerMonk:17} requires $\dim (\ker \mathit{\boldT})^\bot=\infty$ which the authors don't
elaborate on.

The former observation admits to interpret the modified electromagnetic Stekloff eigenvalue problem as asymptotic
limit of the original electromagnetic Stekloff eigenvalue problem for large eigenvalue parameter $\lambda$. Though,
this doesn't yield any non-trivial asymptotic statement on the eigenvalues.

We have seen that (at least in the selfadjoint case) the original electromagnetic Stekloff eigenvalue problem yields
two kind of spectra. Contrary the modified electromagnetic Stekloff eigenvalue problem yields only one kind of spectrum.
This suggests that for inverse scattering applications the original version is more advantageous than the modified
version, because it contains more information. Though the approximation of the modified eigenvalue problem is better
understood than for the original version \cite{Halla:19StekloffAppr}.

It would be further interesting to consider a far field measurement procedure which relates to a second kind of modified
electromagnetic Stekloff eigenvalues. Namely to the spectrum of the asymptotic limit of $A_X(\cdot)$ for small spectral
parameter $\lambda$: $A_{W_1}(\cdot,0)$. This eigenvalue problem can also be formulated as to find
$(\lambda,u)\in \setC\times\{u\in H^1(\Omega)\colon \tr_{\nv\times}\nabla u \in \boldL^2_t(\partial\Omega)\}\setminus\{0\}$
so that
\begin{subequations}
\begin{align}
-\diveps \nabla u &=0 \quad\text{in }\Omega,\\
\nv\cdot\epsilon \nabla u - \lambda\Delta_\partial u &=0 \quad\text{at }\partial\Omega.
\end{align}
\end{subequations}
The special feature of this eigenvalue problem is that it is independent of $\mu$ and $\omega$!

\bibliographystyle{amsplain}
\bibliography{../../../bibliography}

\def\cprime{$'$} \def\cprime{$'$}
\providecommand{\bysame}{\leavevmode\hbox to3em{\hrulefill}\thinspace}
\providecommand{\MR}{\relax\ifhmode\unskip\space\fi MR }
\providecommand{\MRhref}[2]{%
  \href{http://www.ams.org/mathscinet-getitem?mr=#1}{#2}
}
\providecommand{\href}[2]{#2}
\begin{thebibliography}{10}

\bibitem{AmroucheBernardiDaugeGirault:98}
C.~Amrouche, C.~Bernardi, M.~Dauge, and V.~Girault, \emph{Vector potentials in
  three-dimensional non-smooth domains}, Math. Methods Appl. Sci. \textbf{21}
  (1998), no.~9, 823--864. \MR{1626990}

\bibitem{BallCapdeboscqTsering-Xiao:12}
John~M. Ball, Yves Capdeboscq, and Basang Tsering-Xiao, \emph{On uniqueness for
  time harmonic anisotropic {M}axwell's equations with piecewise regular
  coefficients}, Math. Models Methods Appl. Sci. \textbf{22} (2012), no.~11,
  1250036, 11. \MR{2974174}

\bibitem{BoffiBrezziFortin:13}
Daniele Boffi, Franco Brezzi, and Michel Fortin, \emph{Mixed finite element
  methods and applications}, Springer Series in Computational Mathematics,
  vol.~44, Springer, Heidelberg, 2013. \MR{3097958}

\bibitem{BuffaCostabelSheen:02}
A.~Buffa, M.~Costabel, and D.~Sheen, \emph{On traces for {${\bf H}({\bf
  curl},\Omega)$} in {L}ipschitz domains}, J. Math. Anal. Appl. \textbf{276}
  (2002), no.~2, 845--867. \MR{1944792 (2004i:35045)}

\bibitem{CakoniColtonMengMonk:16}
F.~Cakoni, D.~Colton, S.~Meng, and P.~Monk, \emph{Stekloff eigenvalues in
  inverse scattering}, SIAM J. Appl. Math. \textbf{76} (2016), no.~4,
  1737--1763. \MR{3542029}

\bibitem{CakoniColton:06}
Fioralba Cakoni and David Colton, \emph{Qualitative methods in inverse
  scattering theory}, Interaction of Mechanics and Mathematics,
  Springer-Verlag, Berlin, 2006, An introduction. \MR{2256477}

\bibitem{CakoniColtonHaddar:16}
Fioralba Cakoni, David Colton, and Houssem Haddar, \emph{Inverse scattering
  theory and transmission eigenvalues}, CBMS-NSF Regional Conference Series in
  Applied Mathematics, vol.~88, Society for Industrial and Applied Mathematics
  (SIAM), Philadelphia, PA, 2016. \MR{3601119}

\bibitem{CakoniHaddar:09}
Fioralba Cakoni and Houssem Haddar, \emph{On the existence of transmission
  eigenvalues in an inhomogeneous medium}, Appl. Anal. \textbf{88} (2009),
  no.~4, 475--493. \MR{2541136}

\bibitem{CamanoLacknerMonk:17}
Jessika Cama\~no, Christopher Lackner, and Peter Monk, \emph{Electromagnetic
  {S}tekloff eigenvalues in inverse scattering}, SIAM J. Math. Anal.
  \textbf{49} (2017), no.~6, 4376--4401. \MR{3719021}

\bibitem{Costabel:90}
Martin Costabel, \emph{A remark on the regularity of solutions of {M}axwell's
  equations on {L}ipschitz domains}, Math. Methods Appl. Sci. \textbf{12}
  (1990), no.~4, 365--368. \MR{1048563 (91c:35028)}

\bibitem{Dauge:88}
Monique Dauge, \emph{Elliptic boundary value problems on corner domains},
  Lecture Notes in Mathematics, vol. 1341, Springer-Verlag, Berlin, 1988,
  Smoothness and asymptotics of solutions. \MR{961439}

\bibitem{Grisvard:85}
P.~Grisvard, \emph{Elliptic problems in nonsmooth domains}, Monographs and
  Studies in Mathematics, vol.~24, Pitman (Advanced Publishing Program),
  Boston, MA, 1985. \MR{775683}

\bibitem{Halla:19StekloffAppr}
Martin Halla, \emph{Electromagnetic {S}tekloff eigenvalues: approximation
  analysis}, Preprint, 2019, https://arxiv.org/abs/1909.00689.

\bibitem{Halla:19Tcomp}
\bysame, \emph{Galerkin approximation of holomorphic eigenvalue problems: weak
  {T}-coercivity and {T}-compatibility}, Preprint, 2019,
  https://arxiv.org/abs/1908.05029.

\bibitem{Kato:95}
Tosio Kato, \emph{Perturbation theory for linear operators}, Classics in
  Mathematics, Springer-Verlag, Berlin, 1995, Reprint of the 1980 edition.
  \MR{1335452}

\bibitem{KozlovMazya:99}
Vladimir Kozlov and Vladimir Maz{\cprime}ya, \emph{Differential equations with
  operator coefficients with applications to boundary value problems for
  partial differential equations}, Springer Monographs in Mathematics,
  Springer-Verlag, Berlin, 1999. \MR{1729870}

\bibitem{Markus:88}
A.~S. Markus, \emph{Introduction to the spectral theory of polynomial operator
  pencils}, Translations of Mathematical Monographs, vol.~71, American
  Mathematical Society, Providence, RI, 1988, Translated from the Russian by H.
  H. McFaden, Translation edited by Ben Silver, With an appendix by M. V.
  Keldysh. \MR{971506}

\bibitem{Monk:03}
Peter Monk, \emph{Finite element methods for {M}axwell's equations}, Numerical
  Mathematics and Scientific Computation, Oxford University Press, New York,
  2003. \MR{MR2059447 (2005d:65003)}

\bibitem{SanchezSanchez:89}
J.~Sanchez~Hubert and E.~S\'{a}nchez-Palencia, \emph{Vibration and coupling of
  continuous systems}, Springer-Verlag, Berlin, 1989, Asymptotic methods.
  \MR{996423}

\bibitem{Stenberg:95}
Rolf Stenberg, \emph{On some techniques for approximating boundary conditions
  in the finite element method}, vol.~63, 1995, International Symposium on
  Mathematical Modelling and Computational Methods Modelling 94 (Prague, 1994),
  pp.~139--148. \MR{1365557}

\bibitem{Weber:80}
Ch. Weber, \emph{A local compactness theorem for {M}axwell's equations}, Math.
  Methods Appl. Sci. \textbf{2} (1980), no.~1, 12--25. \MR{561375 (81f:78005)}

\end{thebibliography}
\end{document}